\documentclass[11pt, reqno]{amsart}

\usepackage{textcomp} 
\usepackage{amsmath, amsfonts,amsthm,amssymb,amsxtra}
\usepackage{amscd}
\usepackage{enumitem}

\usepackage{xcolor}
\usepackage{graphicx}
\usepackage{verbatim}

\newtheorem{theorem}{Theorem}[section]
\newtheorem{proposition}[theorem]{Proposition}

\newtheorem{corollary}[theorem]{Corollary}
\newtheorem{question}[theorem]{Question}

\theoremstyle{definition}

\newtheorem{definition}[theorem]{Definition}
\newtheorem{examples}[theorem]{Example}

\theoremstyle{remark}

\newtheorem{remark}[theorem]{Remark}

\numberwithin{equation}{section}

\renewcommand{\epsilon}{\varepsilon}

\renewcommand{\phi}{\varphi}

\begin{document}
	
	{
		\title{Simplifying Weinstein Morse functions}
		\author{Oleg Lazarev}
		\address{Oleg Lazarev}
		\address{Columbia University}
		\email{olazarev@math.columbia.edu}  
	} 
	
	\maketitle
	\begin{abstract}
		We prove that the minimum number of critical points of a Weinstein Morse function on a Weinstein domain of dimension at least six is at most two more than the minimum number of critical points of a smooth Morse function on that domain; if the domain has non-zero middle-dimensional homology, these two numbers agree. There is also an upper bound on the number of gradient trajectories between critical points in smoothly trivial Weinstein cobordisms. As an application, we show that the number of generators for the Grothendieck group of the wrapped Fukaya category 
		is at most the number of generators for singular cohomology and hence vanishes for any Weinstein ball.  We also give a topological obstruction to the existence of finite-dimensional representations of the Chekanov-Eliashberg DGA of Legendrian spheres.  		
	\end{abstract}

	\section{Introduction}
	
	Weinstein domains are exact symplectic manifolds equipped with Morse functions compatible with their symplectic structures. These domains encompass a large class of symplectic manifolds, e.g. cotangent bundles, and are closely related to Stein manifolds in complex geometry \cite{CE12}. The Weinstein Morse function gives a symplectic handle-body presentation of the domain and allows one to study its symplectic geometry  via Legendrian knot theory. 
	This handle-body presentation is not unique and, like a smooth handlebody presentation, a Weinstein handle-body presentation can be modified by a series of moves, or   \textit{Weinstein homotopy}, that preserve the symplectic structure of the ambient domain; see Section \ref{sec: background}. In this paper, we study how these moves can be used to simplify an arbitrary Weinstein presentation.

	Abouzaid and Seidel \cite{Abouzaid_Seidel} introduced the \textit{complexity} $WCrit(W)$ of a Weinstein structure $W$ as the minimal number of critical points of a Weinstein Morse function on $W$, up to Weinstein homotopy. The corresponding notion for Stein domains was
	introduced by Eliashberg \cite{Eliashberg_plurisubharmonic_func}.
	Complexity is tautologically a Weinstein homotopy invariant. 
	The analog of $WCrit$ in the smooth setting is $Crit(M)$, the minimal number of critical points of \textit{any} Morse function on a smooth manifold $M$. This is a classical invariant of smooth manifolds and we will study the relationship between $WCrit(W)$ and $Crit(W)$ as a way of investigating the difference between symplectic and smooth topology and the corresponding handle-body moves. 
	
	We first recall some results about $Crit(M)$. A priori $Crit(M)$ is just a smooth invariant of $M$. Morse proved that there is a lower bound for $Crit(M)$ in terms of the integral homology $H_*(M;\mathbb{Z})$. 
	Smale \cite{smale_structure_manifolds} showed in the proof of the h-cobordism theorem that if $M^n$ is simply-connected and $n \ge 6$, then this lower bound is in fact sharp. More precisely, it is possible to simplify an arbitrary Morse function on $M^n$ to another Morse function whose number of critical points agrees with the homological lower bound. So in this case, $Crit(M)$ is actually a homotopy invariant of $M^{n}$. To simplify an arbitrary Morse function, Smale uses 
	certain moves called handle-slides and the Whitney trick, which requires $M^n$ to be simply-connected and $n \ge 6$. In the non-simply-connected case, $Crit(M)$ is not a homotopy invariant. By the s-cobordism theorem, it depends on the choice of an element of the Whitehead group $Wh(\pi_1(M))$. 
	As usual, the situation is different and more complicated in dimension 4. For example, it is unknown whether $Crit(S^4) = 2$ for any smooth structure on $S^4$; by Cerf's theorem \cite{cerf_diff_3sphere}, this question is equivalent to the smooth 4-dimensional Poincare conjecture. 
	
	In this paper, we will study how much of the simplification of smooth Morse functions can be done in the Weinstein setting. Since any Weinstein Morse function is a smooth Morse function, we have the inequality $WCrit(W) \ge Crit(W)$ and Eliashberg \cite{Eliashberg_plurisubharmonic_func} asked whether there are examples where $WCrit(W)$ and $Crit(W)$ differ.
	As first shown by Seidel and Smith \cite{Seidel_Smith_ramanujam_surface}, such examples do exist. For example, $Crit(B^{2n}) =1$ but any Weinstein structure $\Sigma^{2n}$ on $B^{2n}$ that is not symplectomorphic to (the completion of) $B^{2n}_{std}$ must have $WCrit(\Sigma^{2n}) \ge 2$; see Corollary 11.27 of \cite{CE12} or the h-principle for subcritical Weinstein domains.  In fact,  $WCrit(\Sigma^{2n}) \ge 3$ since the Euler characteristic of $B^{2n}$ is $1$. Seidel and Smith constructed such an exotic $\Sigma^{2n}$ and distinguished it from $B^{2n}_{std}$ by the presence of a Floer-theoretically essential Lagrangian torus. 
	Hence the inequality $WCrit(\Sigma) \ge Crit(\Sigma) + 2$ depends crucially on J-holomorphic curve type invariants. From a Weinstein homotopy point of view, $WCrit$ and $Crit$ differ because the Whitney trick, the key part of Smale's proof of the h-cobordism theorem, does not generally work in the symplectic setting; more precisely, smoothly isotopic Legendrian submanifolds are not necessarily Legendrian isotopic.
	
	Given that $Crit, WCrit$ can indeed be different, it is natural to ask how big this difference can be.  We first note that for domains of dimension at least six, there are infinitely many different Weinstein structures in the same \textit{almost} Weinstein class
	\cite{Abouzaid_Seidel, CE12, MM}. So in principle, $WCrit(W)$ can be arbitrarily larger than $Crit(W)$. The first construction of infinitely many exotic Weinstein structures is due to McLean \cite{MM}. He constructed a single exotic ball $\Sigma^{2n}_{1}$ and then 
	showed that $\Sigma^{2n}_{k}:= \natural_{i=1}^k \Sigma^{2n}_{1}$, the boundary connected sum of k copies of  $\Sigma^{2n}_{1}$, are pair-wise non-symplectomorphic, distinguished by a J-holomorphic curve invariant called symplectic homology. In particular, $\Sigma^{2n}_{k}$ has a natural Weinstein presentation with at least $4k-1$ handles ($3k$ handles for $\coprod_{i=1}^k \Sigma_{1}^{2n}$ and $k-1$ index 1 handles). So it was not clear whether these examples have bounded complexity. Later  
	Abouzaid and Seidel \cite{Abouzaid_Seidel} constructed infinitely many exotic Weinstein structures that do have bounded complexity. 
	
	On the other hand, recent work has shown that certain Weinstein structures have minimal complexity, i.e. $WCrit(W) = Crit(W)$. 
	Cieliebak and Eliashberg \cite{CE12} proved that \textit{flexible} Weinstein structures, which satisfy an h-principle that reduces their symplectic topology to the underlying algebraic topology, have minimal complexity. Later Eliashberg, Ganatra, and the author \cite{EGL} constructed infinitely many examples of exotic (non-flexible) Weinstein structures on $T^*S^n$ and showed that they also have minimal complexity. We will show that minimal complexity holds quite generally. 
	
	\subsection{Almost minimal Weinstein presentations}\label{ssec: almost_minimal_subdomains}
	The above examples due to Seidel-Smith and McLean show that in general
	$WCrit(W) \ge Crit(W) +2$. This lower bound comes from J-holomorphic curve invariants (and some mild use of h-principles). 
	Our main result shows that this is the only constraint on $WCrit$. In the following,  we say a smooth domain $W^{2n}$ (with the homotopy type of an n-dim CW complex) is \textit{smoothly critical} if every smooth proper Morse function has a critical point of index $n$; for example, this holds if $H_n(W; \mathbb{Z}) \ne 0$ or $H_{n-1}(W; \mathbb{Z})$ has torsion elements. A smooth domain $W^{2n}$ is \textit{smoothly subcritical} if $W^{2n}$ admits a smooth Morse function all of whose critical points have index strictly less than $n$. 
	A (smoothly subcritical) Weinstein domain is \textit{Weinstein subcritical} if it admits a Weinstein Morse function all of whose  critical points have index strictly less than $n$. Subcritical Weinstein domains are flexible and hence have minimal complexity as mentioned above \cite{CE12}; see Section \ref{subsection: loose} for details. 
	\begin{theorem}\label{thm: fewcrit}
		If $W^{2n}, n \ge 3,$ is a Weinstein domain, then $WCrit(W) \le Crit(W) +2$. Furthermore, if $W$ is smoothly critical, then $WCrit(W) = Crit(W)$. If $W$ is smoothly subcritical and $\pi_1(W) = 0$, then $WCrit(W) = Crit(W)$ if and only if $W$ is a subcritical Weinstein domain; otherwise, $WCrit(W) = Crit(W) + 2$.  
	\end{theorem}
	More precisely, let $WCrit_k(W^{2n})$ denote the minimum number of index $k$ critical points of a Weinstein Morse function on $W^{2n}$; let $Crit_k(W)$ denote the same for a smooth Morse function. Then the proof of Theorem \ref{thm: fewcrit} actually shows that $WCrit_k(W^{2n}) = Crit_k(W^{2n})$ for $k \le n-2$; if $Crit_n(W) \ne 0$, then either $WCrit_{n-1}(W^{2n}) = Crit_{n-1}(W^{2n})$ and $WCrit_n(W^{2n}) = Crit_n(W^{2n})$ or 
	$WCrit_{n-1}(W^{2n}) = Crit_{n-1}(W^{2n}) +1$ 
	and  $WCrit_n(W^{2n}) = 1$. The second case can only happen when $Crit_n(W^{2n}) = 0$, i.e. $W$ is smoothly subcritical. So we always have $WCrit_n(W^{2n}) \le \max\{1, Crit_n(W^{2n})\}$.

	We now explain the assumption that $W^{2n}$ is simply-connected in the smoothly subcritical part of Theorem \ref{thm: fewcrit}; the rest of our results do not require such an assumption. The point is that if $W$ is smoothly subcritical and simply-connected, all of the critical points of any minimizing smooth Morse function necessarily have index less than $n$. So if $WCrit(W) = Crit(W)$, the minimizing Weinstein Morse function must also have all critical points with index less than $n$ and so $W$ is Weinstein subcritical. However this is not true in the non-simply-connected case and it is possible to have $WCrit(W) = Crit(W)$ but for the minimal Weinstein decomposition to still have $n$ handles and not be Weinstein subcritical. 
	For example, we can start with a Weinstein subcritical domain $W_{sub}$ and attach a non-trivial h-cobordism $C$ with handles of index $n-1, n$ so that $WCrit(W_{sub} \cup C) = Crit(W_{sub}\cup C)$; then we can smoothly trade those handles to handles of index $2,3$, which shows that $W_{sub} \cup C$ is smoothly subcritical but not necessarily Weinstein subcritical, e.g. \cite{MM}.

	Now we give some examples illustrating Theorem \ref{thm: fewcrit}.
	\begin{examples}\label{ex: wcrit_cotangent}
		If $M^n, n \ge 3$, is a closed smooth manifold, then $WCrit(T^*M) = Crit(T^*M) \le Crit(M)$ for \textit{any} Weinstein structure on $T^*M$ since it is smoothly critical; if $n\ge 6$ and $\pi_1(M) = 0$, then the second inequality is also an equality. In particular, any Weinstein structure on $T^*S^n$ has $WCrit(T^*S^n) = 2$.
	\end{examples}
	\begin{examples}\label{ex: ball}
		Any Weinstein ball $\Sigma^{2n}$, which is smoothly subcritical with $Crit(\Sigma^{2n})=1$, has either $WCrit(\Sigma^{2n}) = 1$ or $3$. Since $\pi_1(\Sigma^{2n}) = 0$, the structure is Weinstein homotopic to the standard structure $B^{2n}_{std}$ if and only if $WCrit(\Sigma^{2n}) = 1$. In particular, McLean's exotic structures $\Sigma_k^{2n}$, which have natural presentations with at least $4k-1$ critical points, can be Weinstein homotoped to  presentations with just $3$ critical points, corresponding to handles of index $0, n-1$, and $n$. They are all non-standard structures and so $WCrit(\Sigma^{2n}_k) = 3$. 
	\end{examples}

	Our proof of Theorem \ref{thm: fewcrit} relies on Murphy's h-principle for loose Legendrians \cite{Murphy11} (and its consequences for flexible domains) as well as the smooth Whitney trick. Both of these results hold only for $n \ge 3$, hence our restriction on dimension. We do not know whether Theorem \ref{thm: fewcrit} holds if $n = 2$.
	\begin{question}
		Is $WCrit(W^4) \le Crit(W^4) + 2$ for any Weinstein domain $W^4$?
	\end{question}
	We point out that many smooth 4-dimensional domains have only finitely many Weinstein structures, e.g. $B^{4}, T^*T^2$ 
	\cite{CE12, eliashberg_3torus, wendl_fillings}, and it seems unknown whether any 4-dimensional domain admits infinitely many Weinstein structures in the same formal class (unlike in high dimensions when this is always true). A general finiteness result would automatically give a bound of the form $WCrit(W^4) \le Crit(W^4) + C$ for some constant $C$ possibly depending on the diffeomorphism type of $W^4$.

	\subsection{Flexible subdomains}\label{ssec: flex_subdomains}
	
	Our main result Theorem \ref{thm: fewcrit} essentially follows from the following theorem. For a Weinstein domain $W^{2n}, n \ge 3$, let $W_{flex}^{2n}$ be the unique flexible Weinstein structure almost symplectomorphic to $W^{2n}$; see Section \ref{subsection: loose}. 
	\begin{theorem}\label{thm: flexible_subdomain}
		Any Weinstein domain $W^{2n}, n \ge 3$, can be Weinstein homotoped to  $W_{flex}^{2n} \cup C^{2n}$, where $C^{2n}$ is a smoothly trivial Weinstein cobordism with two critical points of index $n-1, n$.  
	\end{theorem}
	This result implies that the smooth topology and the symplectic topology can be separated in the sense that all the smooth topology can be put into a symplectically trivial (flexible) domain while all the symplectic topology can be put into a smoothly trivial cobordism, which is a smooth collar of the boundary of $W^{2n}$. 
	In particular, Theorem \ref{thm: flexible_subdomain} shows that $W_{flex}$ is a Weinstein subdomain of $W$. This extends previous work of Eliashberg and Murphy \cite{EM} who proved that $W_{flex}$ is a \textit{Liouville} subdomain of $W$, i.e. $W\backslash W_{flex}$ is an exact symplectic cobordism, perhaps without a compatible Weinstein Morse function. The decomposition in Theorem \ref{thm: flexible_subdomain} has several applications, which are explored in \cite{Lazarev_maximal, Lazarev_reg_lag}; for example, it is used to prove an existence h-principle for regular Lagrangians with boundary in arbitrary Weinstein domains and construct `maximal' Weinstein domains.
	
	Theorem \ref{thm: flexible_subdomain} immediately implies most of Theorem \ref{thm: fewcrit}. The presentation in Theorem \ref{thm: flexible_subdomain} shows that $WCrit(W) \le WCrit(W_{flex}) + 2$. Since flexible structures have minimal complexity \cite{CE12},  $WCrit(W_{flex}) = Crit(W)$. Combining these results, we get $WCrit(W) \le Crit(W) + 2$, the first claim in Theorem \ref{thm: fewcrit}. 
	The proof of the smoothly critical case of Theorem \ref{thm: fewcrit} is similar. Finally, we note that the existence of such a decomposition does not contradict the above mentioned fact that certain Weinstein domains require more than one generator for their wrapped Fukaya category. This is because the flexible co-cores of the flexible domain become non-flexible once the last handle is attached.

	Flexible Weinstein domains are defined only for $n \ge 3$. The analog of these domains for $n = 2$ are Weinstein domains whose index 2 handles are attached along stabilized Legendrians; we will call these \textit{stabilized domains}. However, neither stabilized Legendrians nor  stabilized domains satisfy an h-principle and so we do not know whether Theorem \ref{thm: fewcrit} or Corollary \ref{cor: generators_fukaya_cat} hold for $n = 2$. 
	However  versions of some of our results, like Theorem \ref{thm: flexible_subdomain}, Theorem \ref{thm: flexdomainhandle}, and Corollary \ref{cor: non_loose_link}, continue to hold for $n = 2$ if we replace flexible domains, loose Legendrians with these analogous domains, Legendrians respectively. For example, we have the following version of Theorem \ref{thm: flexible_subdomain}. 
	\begin{theorem}\label{thm: dim4_sh}
		Any Weinstein domain $W^4$ can be  Weinstein homotoped to $V^4 \cup H^2$, where $V^4$ is a stabilized domain that is simply homotopy equivalent to $W^4 \cup H^1$.
	\end{theorem}
    The notation $H^n_{\Lambda}$ denotes a Weinstein handle attached along an isotropic attaching sphere $\Lambda$
   and we write $H^n$ if we not specify the attaching sphere; see Section \ref{sec: background}. 
    We note that the stronger version of Theorem \ref{thm: flexible_subdomain} is false for $n =2$: in general, there may not exist a stabilized $V^4$ that is \textit{diffeomorphic} to $W^4 \cup H^1$. 
	For example, there is a unique Weinstein structure on $T^*T^2$ and it has non-vanishing symplectic homology \cite{eliashberg_3torus, wendl_fillings}; the same holds for $T^*T^2 \cup H^1$ \cite{CE12}.
	On the other hand, stabilized domains have vanishing symplectic homology and so  $T^*T^2 \cup H^1$ does not admit a stabilized Weinstein structure.

	Theorem \ref{thm: flexible_subdomain} shows that any Weinstein domain $W^{2n}, n \ge 3$, can be presented as a flexible domain $W_{flex}^{2n} \cup H^{n-1}$ plus a single critical handle. In fact, the proof of Theorem \ref{thm: flexdomainhandle} is a bit more explicit about the single extra handle.  
	\begin{corollary}\label{cor: non_loose_link}
		Every Weinstein domain $W^{2n}, n \ge 3$, can be Weinstein homotoped to a subcritical domain $V_{sub}$ with handles attached to the Legendrian link 
		$\Lambda_1 \coprod \cdots \coprod \Lambda_{k-1} \coprod  \Lambda_k \subset \partial V_{sub}$ such that $\Lambda_1 \coprod \cdots \coprod \Lambda_{k-1}$ is a loose link and $\Lambda_k$ is a loose Legendrian. 
	\end{corollary}
	Even though all of the Legendrians in Corollary \ref{cor: non_loose_link} are individually loose, the entire link $\Lambda_1 \coprod \cdots \coprod \Lambda_{k-1} \coprod \Lambda_k$ 
	may not be loose, i.e. the loose charts of $\Lambda_i$ intersect $\Lambda_k$ and loose chart of $\Lambda_k$ intersects $\Lambda_i$.  
	Otherwise all Weinstein domains would be flexible. 
	So the attaching Legendrians are themselves symplectically trivial but their linking is symplectically non-trivial, i.e. the symplectic topology of the domain is captured in this linking. Of course, $\Lambda_k$ becomes non-loose once we attach handles to $\Lambda_{1}, \cdots, \Lambda_{k-1}$ (and vice-versa).

	Now we present an example demonstrating Theorem \ref{thm: flexible_subdomain}. 
	\begin{examples}\label{ex: TS^n_decomp}
		Any Weinstein structure on $T^*S^n, n \ge 3,$ can be Weinstein homotoped to $T^*S^n_{flex} \cup H^{n-1} \cup H^n_\Lambda$ for some Legendrian $\Lambda$ in the contact manifold $\partial (B^{2n}_{std} \cup H^{n-1})$. A slightly modified version of Theorem \ref{thm: flexible_subdomain} shows that $T^*S^n$ can also be homotoped to $B^{2n}_{std} \cup H^n_{\Lambda}$; this is why we always have $WCrit(T^*S^n) = 2$ in Example \ref{ex: wcrit_cotangent}. We can reformulate this as follows. Let  $\mathfrak{Legendrian}((Y, \xi); \Lambda_0)$ denote parametrized Legendrians in the contact manifold $(Y, \xi)$, up to Legendrian isotopy, 
		that are in some fixed Legendrian formal isotopy class $\Lambda_0$. Let  $X^{2n}$ be an almost Weinstein domain, i.e. an almost complex domain with the homotopy type of an $n$-dimensional CW complex;  see Section \ref{sec: background}. Then let $\mathfrak{Weinstein}(X^{2n})$ denote Weinstein structures on $X^{2n}$ up to Weinstein homotopy. There is a natural map 
		\begin{equation}\label{eqn: leg_handle}
		\mathcal{H}_{crit}:\mathfrak{Legendrian}((S^{2n-1}, \xi_{std}); \Lambda_{unknot})
		\rightarrow 
		\mathfrak{Weinstein}(T^*S^n)
		\end{equation}
		taking a Legendrian $\Lambda \subset (S^{2n-1}, \xi_{std})= \partial B^{2n}_{std}$ which is formally isotopic to $\Lambda_{unknot}$   to the Weinstein structure $B^{2n}_{std} \cup H^n_{\Lambda}$ on $T^*S^n$.  The statement that  $WCrit = 2$ for any Weinstein structure on $T^*S^n$ implies that this map is surjective. In particular, the class of connected Legendrians is as complicated as the class of Weinstein structures. It is known that there are infinitely many Weinstein structures on $T^*S^n$ \cite{Abouzaid_Seidel, EGL, MM}, each distinguished by symplectic homology. Hence this reproves the result that there are infinitely many Legendrians in the same formal class as the Legendrian unknot; see Remark 4.11 in \cite{EGL}.
	\end{examples}

	Although our main result shows that Weinstein homotopy moves are more flexible than they might seem, there are limits to this flexibility. 
	For example,  Theorem \ref{thm: flexible_subdomain} shows that any Weinstein domain can be presented as a flexible domain plus a single extra handle, which is possibly non-flexible. As we now explain, it is crucial that the non-flexible critical handle is attached last and in general, it is impossible to first attach non-flexible handles and then attach flexible handles. So \textit{order} of flexilibity/non-flexibility matters, which is a sign of rigidity. As  expected, this rigidity ultimately comes from J-holomorphic curves.
	\begin{examples}\label{ex: TS^n_order}
		By Theorem \ref{thm: flexible_subdomain}, $T^*S^n_{std}$ is Weinstein homotopic to $T^*S^n_{flex} \cup H^{n-1} \cup H^n_{\Lambda} = (B^{2n}_{std} \cup H^n_{flex}) \cup H^{n-1} \cup H^n_{\Lambda} $ for some Legendrian $\Lambda$. In this case, we attach flexible handles first and then non-flexible handles. 
		However,  $T^*S^n_{std}$ cannot be presented as 
		$(B^{2n}_{std} \cup H^{n-1} \cup H^n_{\Lambda}) \cup H^n_{flex}$, where we first attach non-flexible handles and then flexible handles. This presentation is equivalent to a Weinstein structure  of the form $\Sigma^{2n}\cup H^n_{flex}$, for some exotic ball $\Sigma^{2n}$. We claim that $T^*S^n_{std}$ is not symplectomorphic to  $\Sigma^{2n}\cup H^n_{flex} $ for any $\Sigma^{2n}$.
		To see this, let $C \subset \Sigma^{2n}\cup H^n_{flex}$ be the Lagrangian co-core of $H^n_{flex}$. Since $H^n_{flex}$ is attached along a loose Legendrian in $\partial \Sigma^{2n}$,  the wrapped Floer homology $WH(C, C; T^*S^n_{std})$ vanishes. But $C$ generates $H_n(T^*S^n, \partial T^*S^n)  \cong \mathbb{Z}$ and so $C\cdot S^n = 1$, where $S^n \subset T^*S^n_{std}$ is the zero-section, a closed exact Lagrangian. But $WH(C, C; T^*S^n_{std}) = 0$ implies that $WH(C, S^n; T^*S^n_{std}) = 0$ and so $C \cdot S^n = \chi(WH(C, S^n; T^*S^n_{std}))= 0$, a contradiction. Another related way to see that $T^*S^n_{std}$ and $\Sigma^{2n}\cup H^n_{flex}= \Sigma^{2n} \natural T^*S^n_{flex}$ are not symplectomorphic is to note that the Grothendieck groups of their wrapped Fukaya categories are different: 	
		$K_0(\mathcal{W}(T^*S^n_{std})) \cong \mathbb{Z}$ while 
		$K_0(\mathcal{W}(T^*S^n_{flex} \natural \Sigma^{2n})) \cong 0$; see Section \ref{sec: fukaya_CE}.

		Since $T^*S^n_{std}$ is not of the form $\Sigma^{2n} \cup H^n_{flex}$, the map 
		\begin{equation}\label{eqn: non_surjective_loose_handle_attachment}
		\mathcal{H}_{loose}:\mathfrak{Weinstein}(B^{2n}) \rightarrow \mathfrak{Weinstein}(T^*S^n)
		\end{equation} 
		obtained by attaching a critical handle along a \textit{loose} Legendrian unknot to an exotic Weinstein ball is not surjective. This map is well-defined since any contact structure $\partial \Sigma^{2n}$ in the almost contact structure $(S^{2n-1}, J_{std})$ 
		has a unique loose Legendrian in the standard formal class. Furthermore, it has infinite image; for example, $\mathcal{H}_{loose}$ is injective on the exotic structures $\Sigma^{2n}_k$ constructed by McLean \cite{MM}. We contrast the non-surjectivity of $\mathcal{H}_{loose}$, a rigidity result, to the surjectivity of the map $\mathcal{H}_{crit}$ in Equation \ref{eqn: leg_handle}, a flexibility result. 
	\end{examples}

	Now we sketch the proof of Theorem \ref{thm: flexible_subdomain}, which implies the main result Theorem \ref{thm: fewcrit}.  The key idea is that certain  Weinstein homotopy moves called handle-slides can be used to make a Legendrian loose; see Section \ref{sec: background}. 
	More precisely, given two Legendrians and a local chart intersecting them, the handle-slide produces another Legendrian, which was described by Casals and Murphy \cite{Casals_Murphy_front}. We will show that there is a special choice of local chart such that the handle-slid Legendrian is loose (not all choices of charts result in loose Legendrians). For an arbitrary Weinstein domain, we  fix one Legendrian and handle-slide the rest of the Legendrians over that fixed Legendrian. For appropriate choices of local charts, the resulting Legendrians  form a loose link except for the fixed Legendrian which will in general intersect the loose charts of the other Legendrians; this is the content of Theorem \ref{thm: flexible_subdomain}.
	
	\subsection{Gradient trajectories and Reeb chords}\label{ssec: grad_traj}
	As mentioned before, one of our goals is to study to what extent the  simplification of smooth Morse functions holds in the Weinstein setting.
	As we explained before, the simplification in the smooth case was done by Smale \cite{smale_structure_manifolds} in the h-cobordism theorem, whose proof has two main steps. The first step is to apply handle-slides to make handles with consecutive indices cancel algebraically, i.e. for the belt sphere of a $k$ handle and the attaching sphere of a $k+1$ handle to have algebraic intersection number one. From the Morse theory point of view, the intersection of the belt sphere of a $k$-handle and the attaching sphere of a $k+1$ handle correspond to gradient trajectories between the associated index $k, k+1$ critical points. The second step is to use the Whitney trick to reduce the number of intersection points between algebraically cancelling handles to make them geometrically cancelling, i.e. have \textit{geometric} intersection number one. 
	
	Since Weinstein handles can be handle-slid in much the same way as smooth handles, the first step can be done in the Weinstein setting. However the second step necessarily fails since $WCrit(W) \ne Crit(W)$ in general but we can try attempt to perform it and see how far we get. 
	By Theorem \ref{thm: flexible_subdomain}, any smoothly trivial Weinstein cobordism $W$ can be Weinstein homotoped to have two Weinstein handles of index $n-1, n$ that cancel algebraically, i.e. $W = H^{n-1} \cup H^n_{\Lambda}$.  The Whitney trick shows that in this case, it is possible to \textit{smoothly} isotope the attaching sphere $\Lambda$ so it intersects the belt sphere of $H^{n-1}$ in exactly one point. However, if  $\Lambda$ intersects the belt sphere of $H^{n-1}$ in a single point, then it is loose \cite{CE12}, which implies that the Weinstein cobordism is flexible. Hence, in general it is impossible to realize this smooth isotopy by a Legendrian isotopy and to reduce the geometric intersection number to one. The minimal possible number is therefore three; it must be greater than one and must be odd since it agrees mod $2$ with the algebraic intersection number, which is  one. Although we do not know whether the geometric intersection number can always be reduced to three, in the following result we show that it is possible to reduce this number to some universal constant independent of the Weinstein structure. 
	So we can get uniformly close to realizing the second step of Smale's h-cobordism proof.  
	\begin{theorem}\label{thm: intersection points}
		There exists a constant $C_n \ge 3$ depending only on $n$ such that 
		any smoothly trivial Weinstein cobordism $W^{2n}, n \ge 3,$ can be Weinstein homotoped to a presentation with two handles of index $n-1, n$ such that the belt sphere of the $n-1$ handle and the attaching sphere of the $n$ handle intersect $C_n$ times. 
	\end{theorem}
	This is equivalent to having a Weinstein Morse function with two critical points of index $n-1, n$ such that there are $C_n$ gradient trajectories from the index $n$ to the index $n-1$ critical point. The proof of Theorem \ref{thm: intersection points} actually shows that it is possible in principle to compute $C_n$. However this seems to depend on having a good understanding of a certain (local) Legendrian isotopy which comes from an h-principle and is therefore not very explicit. 
	
	We also point out that Theorem \ref{thm: intersection points} can be interpreted as a decomposition of Legendrian attaching spheres for smoothly trivial Weinstein cobordisms. For example, suppose $\partial_-W^{2n} = (S^{2n-1}, \xi_{std})$ in Theorem \ref{thm: intersection points}. Up to Weinstein homotopy, we can assume that the attaching Legendrian has a standard part that intersects the belt sphere $C_n$ times and a variable part outside the belt sphere, i.e. in $(S^{2n-1}, \xi_{std}) \backslash Op(S^{n-2}) \subset (\mathbb{R}^{2n-1}, \xi_{std})$. So  all the interesting symplectic topology of the attaching Legendrian can be put outside the belt sphere, in $(\mathbb{R}^{2n-1}, \xi_{std})$
	(but of course it is important that the attaching Legendrian interact non-trivially with the belt sphere). Put another way, any Weinstein structure on $B^{2n}$ can be obtained from $B^{2n}_{std}$
	by attaching a single ``generalized" Weinstein handle along a singular Legendrian with pinwheel singularity with at most $C_n$ spokes on the pinwheel.

	Theorem \ref{thm: flexible_subdomain} shows that all the interesting symplectic topology of a Weinstein domain occurs in the interaction of two smoothly cancelling handles of index $n-1$ and $n$ and Theorem \ref{thm: intersection points} shows that it is possible to simplify this interaction.
	However in the presence of multiple $n-1$ handles, the attaching Legendrian for the $n$-handle might have to pass through \textit{all} $n-1$ handles, even when this is topologically unnecessary, showing that in general the situation is more complicated than these results might seem to indicate. Again we need to use J-holomorphic curve invariants for such a rigidity statement. 
	\begin{examples}
		Consider a subflexible Weinstein structure $W^{2n}$ on $B^{2n}\cup H^{n-1}$ that is not flexible. Such an example was constructed by Murphy and Siegel \cite{MS} and has zero symplectic homology $SH(W^{2n})$ but non-zero \textit{deformed} symplectic homology $SH^{\alpha}(W^{2n})$; 
		here $\alpha$ is the generator of $H^{n-1}(B^{2n}\cup H^{n-1})\cong \mathbb{Z}$.  
		So this domain is smoothly subcritical but is not symplectically subcritical and hence by Theorem \ref{thm: fewcrit} admits a Weinstein presentation with four handles: one of index 0, two of index n-1, and one of index n, i.e. $B^{2n}_{std} \cup H^{n-1}_1 \cup H^{n-1}_2 \cup H^n_{\Lambda}$.  Here $\Lambda$ has algebraic intersection number 1 with $H^{n-1}_1$ and 0 with $H^{n-1}_2$. However $\Lambda$ has geometric intersection number at least 3 with $H^{n-1}_1$ since otherwise $\Lambda$ would be loose. 
		Furthermore,  $\Lambda$ must have geometric intersection number at least 2 with $H^{n-1}_2$; so $\Lambda$ must interact with \textit{both} $H^{n-1}_1$ and $H^{n-1}_2$. 
		Otherwise, the domain would be of the form $(B^{2n}_{std}\cup H^{n-1}_1 \cup H^{n}_{\Lambda}) \cup H^{n-1}_2 = \Sigma^{2n} \cup H^{n-1}$, for some exotic structure $\Sigma^{2n}$ on $B^{2n}$.  However $\Sigma^{2n} \cup H^{n-1}$ has zero deformed symplectic homology as we now show. Since $H^{n-1}$ is a subcritical handle, the Viterbo transfer map $SH^{\alpha}(\Sigma^{2n} \cup H^{n-1}) \rightarrow SH^{i^*\alpha}(\Sigma^{2n})$ is an isomorphism, where $i^*: H^{n-1}(\Sigma^{2n} \cup H^{n-1}) \rightarrow H^{n-1}(\Sigma^{2n})$ is the induced map on cohomology. 
		Since $i^*\alpha  \in H^{n-1}(\Sigma^{2n}) = 0$, $SH^{i^*\alpha}(\Sigma^{2n})$ agrees with the undeformed symplectic homology $SH(\Sigma^{2n})$. Since $\Sigma^{2n}$ is a subdomain of $W^{2n}$, which has vanishing $SH$, and the Viterbo map is unital, $SH(\Sigma)$ also vanishes.  Therefore  $SH^{\alpha}(\Sigma^{2n} \cup H^{n-1})$ is also zero and so $\Sigma^{2n} \cup H^{n-1}$ cannot be Weinstein homotopic to $W^{2n}$. We note that Theorem \ref{thm: intersection points} shows that it is possible to choose $\Lambda$ so that it intersects the belt sphere of $H^{n-1}_1$ at most $C_n$ times (since $W^{2n} \backslash (B^{2n}_{std} \cup H^{n-1}_2)$ is smoothly trivial). We do not know whether there is an analogous bound on the intersection number between $\Lambda$ and $H^{n-1}_2$. 
		
		Since $W$ is not of the form $\Sigma^{2n} \cup H^{n-1}$ for any exotic Weinstein ball $\Sigma^{2n}$,  the map 
		\begin{equation}\label{eqn: non_surjective_subcritical_handle}
		\mathcal{H}_{sub}: \mathfrak{Weinstein}(B^{2n}) \rightarrow \mathfrak{Weinstein}(B^{2n} \cup H^{n-1})
		\end{equation}
		obtained by attaching a \textit{subcritical} handle to an exotic Weinstein ball is not surjective; see \cite{Ghiggini_subcritical_contact_surgery} for an analogous result in the contact case. This rigidity result is similar to the non-surjectivity of the map $\mathcal{H}_{loose}$ in Equation \ref{eqn: non_surjective_loose_handle_attachment} for flexible handle attachment and in contrast to the surjectivity of $\mathcal{H}_{crit}$ in Equation \ref{eqn: leg_handle} for critical handle attachment to the standard Weinstein ball. 
	\end{examples}
	
		\subsection{Results for the wrapped Fukaya category and the Chekanov-Eliashberg DGA}\label{sec: fukaya_CE} 
	We now give some applications of the flexibility results in Sections \ref{ssec: almost_minimal_subdomains}, \ref{ssec: flex_subdomains} 
to certain J-holomorphic curve invariants.
	  To a Weinstein (or Liouville) domain $X^{2n}$ (with a choice of grading data), one can associate the wrapped Fukaya category $\mathcal{W}(X)$ of $X$,  a certain $A_\infty$-category. 
	The objects of $\mathcal{W}(X)$ are (graded) exact Lagrangians in $X^{2n}$ that are closed or have Legendrian boundary in $\partial X^{2n}$; the morphisms are wrapped Floer cochains.
	In the context of homological mirror symmetry, one considers the derived Fukaya category $D^b\mathcal{W}(X)$, the cohomology category of twisted complexes over $\mathcal{W}(X)$, i.e. 
	$D^b\mathcal{W}(X): = H^0(Tw(\mathcal{W}(X))$.
	This is a triangulated category and Weinstein domains with symplectomorphic completions have exact equivalent derived Fukaya categories. 
	
	To obtain a more explicit description of the wrapped Fukaya category, 
	it is useful to find a set of \textit{generators}. 
	Since $D^b\mathcal{W}(X)$ is triangulated, one can take mapping cones on morphisms. A set of objects $G_i$ are  generators of $D^b\mathcal{W}(X)$  if every object of the category is isomorphic to an iterated mapping cone on them;  equivalently, every object is  isomorphic to a twisted complex on the generators. In this case, there is an exact equivalence between $D^b\mathcal{W}(X)$ and $H^0(Tw(\mathcal{G}))$, where $\mathcal{G}$ is the $A_\infty$-subcategory with objects $G_i$.
	Let $g(\mathcal{W}(X))$ denote the minimum number of generators for $D^b\mathcal{W}(X)$. 
	Many proofs of homological mirror symmetry involve finding some collection of generators for $D^b\mathcal{W}(X)$ and then showing that the endomorphism algebra of these generators is quasi-isomorphic to the endomorphism algebra of some generating coherent sheaves on the mirror.
	
	Theorem \ref{thm: fewcrit} can be used to bound the number of generators $g(\mathcal{W}(X))$ for  $D^b\mathcal{W}(X)$. The unstable manifold of an index $n$ critical point of a Weinstein Morse function, or \textit{co-core}, is a Lagrangian disk with Legendrian boundary and hence defines an object in $D^b\mathcal{W}(X)$.  As proven in \cite{chantraine_cocores_generate, ganatra_generation}, the co-cores of the index $n$ critical points of any Weinstein Morse function on $X$ generate  $D^b\mathcal{W}(X)$, i.e.
	$g(\mathcal{W}(X^{2n})) \le WCrit_n(X^{2n})$. 	
	Theorem \ref{thm: fewcrit} shows that there is a topological bound on $WCrit_n(X^{2n})$ and hence on the number of generators needed. For the following result, let $g(A)$ denote the minimum number of generators of an abelian group $A$. 
	\begin{corollary}\label{cor: generators_fukaya_cat}
		If $X^{2n}, n \ge 3$, is a Weinstein domain, then $g(\mathcal{W}(X)) \le \max\{1, g(H^n(X; \mathbb{Z})) \}$.
	\end{corollary}
	\begin{proof}
		The proof of Theorem \ref{thm: fewcrit} shows that $WCrit_n(X) \le \max\{1, Crit_n(X)\}$ for all $X^{2n}$. 
		Combining this with the result from \cite{chantraine_cocores_generate, ganatra_generation}, we get the inequality 
		$g(\mathcal{W}(X)) \le \max\{1, Crit_n(X) \}$.
		If $X^{2n}$ is simply-connected, then  Smale's h-cobordism theorem (which holds since $n \ge 3$) implies that 
		$Crit_n(X)  =  g(H^n(X; \mathbb{Z}))$, which proves the result in that case. If $X^{2n}$ is not simply-connected, we attach some $2$-handles to $X^{2n}$ to get a simply-connected Weinstein domain $Y^{2n}$. Since $n \ge 3$, we have 
		$H^n(Y^{2n}; \mathbb{Z}) \cong 
		H^n(X^{2n}; \mathbb{Z})$ and so 
		$g(H^n(Y^{2n}; \mathbb{Z})) = 
		g(H^n(X^{2n}; \mathbb{Z}))$.
		Furthermore, since $n \ge 3$, the $2$-handles are subcritical and hence 
		$D^b\mathcal{W}(Y)$ is exact equivalent to 
		$D^b\mathcal{W}(X)$ by \cite{ganatra_generation} and so $g(\mathcal{W}(X)) = g(\mathcal{W}(Y))$. Then the result for $Y^{2n}$, which is simply-connected, implies the result for $X^{2n}$. 	
	\end{proof}
	\begin{remark} 
		Since the Lagrangian co-cores are disks, they are graded objects for any grading of the wrapped Fukaya category. As noted in \cite{chantraine_cocores_generate}, generation by co-cores holds for any grading of the wrapped Fukaya category and therefore our results also hold for any grading. 
	\end{remark} 
	
	A related notion is that of \textit{split-generation}: a set of objects are split-generators if every objects of the category is a \textit{summand} of a twisted complex on these objects. This is a useful notion since there are closed symplectic manifolds whose Fukaya categories have finitely many split-generators but no finite collection of generators, e.g. the 2-torus. 	
	We emphasize that Corollary \ref{cor: generators_fukaya_cat} concerns generation, not split-generation.  
	Whenever there is a finite collection of generators (or split-generators), there is a single split-generator, namely the direct sum of all these objects. So the number of split-generators is not an interesting invariant. 
	
	The number of generators, on the other hand, is a meaningful invariant and in certain cases, the inequality in Corollary \ref{cor: generators_fukaya_cat} is sharp. For example, if $X^{2n}$ is a Weinstein ball, then Corollary \ref{cor: generators_fukaya_cat} shows that at most one generator is needed and if the Fukaya category of this ball is non-trivial (as is the case for the exotic structures constructed by McLean \cite{MM}), then at least one generator is needed. In certain cases, the number of generators needed for $\mathcal{W}(X)$ is greater than one. Since   $D^b\mathcal{W}(X)$ is a triangulated category, we can consider its Grothendieck group $K_0(\mathcal{W}(X)): = K_0(D^b\mathcal{W}(X))$.
	For any triangulated category, the minimum number of generators for the Grothendieck group
	gives a lower bound on the number of generators of the category. In particular, Corollary \ref{cor: generators_fukaya_cat} 
	implies 
	that  for any Weinstein domain $X^{2n}, n \ge 3,$ we have
	\begin{equation}\label{eqn:grot_inequality}
	g(K_0(\mathcal{W}(X))) \le g(\mathcal{W}(X)) \le  \max\{1, g(H^n(X^{2n}; \mathbb{Z})) \}
	\end{equation}
	There are Weinstein domains for which $g(K_0(\mathcal{W}(X)))$ is bigger than one.  For example, consider 
	the boundary connected sum $\natural^k T^*S^n$ of $k$ copies of $T^*S^n_{std}$. As explained to the author by Abouzaid,  $K_0(\mathcal{W}(\natural^k T^*S^n))$ has rank at least $k$. Namely, let $\phi_i: K_0(\mathcal{W}(\natural^k T^*S^n)) \rightarrow \mathbb{Z}$ be $\chi(HW(\_, S_i^n))$,
	the Euler characteristic of morphisms from the $i$th-zero section $S_i^n$. Then $(\phi_1, \cdots, \phi_k): K_0(\mathcal{W}(\natural^k T^*S^n)) \rightarrow \mathbb{Z}^k$ is surjective and so $g(K_0(\mathcal{W}(\natural^k T^*S^n))) \ge k$. On the other hand, $g(H^n(\natural^k T^*S^n; \mathbb{Z})) = k$ and so all the inequalities in Equation \ref{eqn:grot_inequality} are all actually equalities. 	
	If we consider generators for the Grothendieck group instead of for the  whole category, a stronger version of
	Corollary \ref{cor: generators_fukaya_cat} holds. 
	The following corollary of Equation \ref{eqn:grot_inequality} was explained to the author by Ivan Smith in the case when $X^{2n}$ is a ball.
	\begin{corollary}\label{cor: Grothendieck_group}
		If $X^{2n}, n \ge 3$ is a Weinstein domain, then $g(K_0(\mathcal{W}(X))) \le g(H^n(X; \mathbb{Z}))$. 
		In particular, if $H^n(X; \mathbb{Z}) = 0$, then $K_0(\mathcal{W}(X))  = 0$. 	
	\end{corollary}
	\begin{proof}
		The case $	 g(H^n(X; \mathbb{Z})) \ge 1 $ is proven by Equation \ref{eqn:grot_inequality} so it suffices to do the case when  $g(H^n(X; \mathbb{Z})) = 0$. Then $g(K_0(\mathcal{W}(X)) ) \le 1$ by Equation \ref{eqn:grot_inequality} and if 
		$g(K_0(\mathcal{W}(X))) = 0$, we are done. Otherwise,
		$g(K_0(\mathcal{W}(X))) = 1$ and so 
		$K_0(\mathcal{W}(x)) \cong \mathbb{Z}/k \mathbb{Z}$ for some integer $k \ge 0$. 
		Now we take the boundary connected sum and form the new Weinstein domain $X\natural X$.
		Since $1$-handles are subcritical,  $D^b\mathcal{W}(X\natural X) \cong D^b\mathcal{W}(X \coprod X)$ by \cite{ganatra_generation} and  $D^b\mathcal{W}(X \coprod X) \cong  D^b\mathcal{W}(X) \prod
		D^b\mathcal{W}(X)$. 		 
		As a result, 	$K_0(\mathcal{W}(X\natural X)) \cong K_0(\mathcal{W}(X)) \oplus K_0(\mathcal{W}(X)) \cong \mathbb{Z}/k \mathbb{Z} \oplus \mathbb{Z}/k \mathbb{Z}$. This implies that
		$g(K_0(\mathcal{W}(X\natural X))) = 2$ since $\mathbb{Z}/k \mathbb{Z} \oplus \mathbb{Z}/k \mathbb{Z}$ is not a cyclic group. 
		On the other hand, we also have $H^n(X \natural X; \mathbb{Z}) \cong 
		H^n(X; \mathbb{Z})
		\oplus 
		H^n(X; \mathbb{Z}) \cong 0$ and so 
		$g(H^n(X \natural X; \mathbb{Z}))= 0$. 
		Again using the previous inequality, 
		we get that 
		$g(K_0(\mathcal{W}(X\natural X))) \le 1$, which contradicts 
		$g(K_0(\mathcal{W}(X\natural X))) = 2$. 
		Therefore, we must have that  $g(K_0(\mathcal{W}(X))) = 0$ and so 
		$K_0(\mathcal{W}(X)) = 0$ as desired. 
	\end{proof} 
	
	In particular, any Weinstein ball $\Sigma^{2n}$ must have 
	$K_0(\mathcal{W}(\Sigma)) = 0$. On the other hand, there are many exotic Weinstein balls $\Sigma^{2n}$ with non-zero symplectic homology \cite{MM}. So their wrapped Fukaya categories are examples of  triangulated categories with non-zero Hochschild cohomology but zero Grothendieck group. We also note that there are examples where the inequality in Corollary \ref{cor: Grothendieck_group} is sharp, e.g.   $\natural^k T^*S^n_{std}$. Conversely, for any integer $j \le k = g(H^n(\natural^k T^*S^n; \mathbb{Z}))$, there is a Weinstein structure $X_j^{2n}$ on $\natural^k T^*S^n$ such that 
	$g(K_0(\mathcal{W}(X_j))) = j$: take $X_j^{2n}$ to be the boundary connected sum of $j$ copies of the standard structure $T^*S^n_{std}$ and $n-j$ copies of the flexible structure $T^*S^n_{flex}$.

	One natural question is what triangulated categories can arise as the wrapped Fukaya category of Weinstein domains. For example, 
	the wrapped Fukaya category of a Weinstein domain is a smooth category with a non-compact Calabi-Yau structure \cite{chantraine_cocores_generate, Ganatra_thesis}.  
	Corollary \ref{cor: Grothendieck_group} further restricts which categories can arise as the Fukaya categories of Weinstein domains and shows that in general the answer depends on the smooth topology of the domain. For example, we have the following result. 
	\begin{corollary} 
		There is no Weinstein ball $\Sigma^{2n}$ such that $D^b(\mathcal{W}(\Sigma^{2n}))$ is exact equivalent  $D^b(\mathcal{W}(T^*S^n_{std}))$. There is no Weinstein  structure $X^{2n}$ on $T^*S^n$  such that 
		$D^b(\mathcal{W}(X^{2n}))$ is exact equivalent to 
		$D^b(\mathcal{W}(T^*S^n_{std} \natural T^*S^n_{std}))$. 
	\end{corollary} 
	\begin{proof}
		As noted above, $g(K_0(\mathcal{W}(T^*S^n_{std}))) = 1$ and 
		$g(K_0(\mathcal{W}(T^*S^n_{std} \natural T^*S^n_{sdtd}) ) ) =2$. However if $\Sigma^{2n}$ is a ball, 
		$g(K_0(\mathcal{W}(\Sigma^{2n}))) = 0$; if  $H^n(X; \mathbb{Z}) \cong \mathbb{Z}$,  $g(K_0(\mathcal{W}(X))) \le 1$. 
	\end{proof}
	
	On the other hand, for any Weinstein ball $\Sigma^{2n}$, the Weinstein domain $T^*S^n_{flex} \natural \Sigma^{2n}$ is a Weinstein structure on $T^*S^n$  with the same wrapped Fukaya category as $\Sigma^{2n}$. Hence the class of categories arising as Fukaya categories of Weinstein structures on $T^*S^n$ is genuinely larger than that for a ball $B^{2n}$. Similarly, for any Weinstein structure $X^{2n}$ on $T^*S^n$, the boundary connected sum $T^*S^n_{flex} \natural X^{2n}$ is a Weinstein structure on $T^*S^n \natural T^*S^n$ with the same wrapped Fukaya category as $X^{2n}$.

	Since Weinstein domains are constructed by attaching handles along Legendrians, Corollary \ref{cor: Grothendieck_group} has implications for J-holomorphic curve invariants of Legendrians.  Given a Legendrian sphere $\Lambda^{n-1}$ in a contact manifold $(Y^{2n-1}, \xi)$ with a Weinstein filling $W^{2n}$, there are (at least) two associated Legendrian isotopy invariants: the Chekanov-Eliashberg algebra $CE(\Lambda)$ of $\Lambda$ (augmented by the filling $W^{2n}$) and the wrapped Floer cochains $CW(C, C)$ of the co-core $C^n$ of the Weinstein $n$-handle $H^n_{\Lambda}$ in the Weinstein domain $W^{2n} \cup H^n_{\Lambda}$. For both invariants, we work over a common ground field $\mathbb{K}$. The former invariant is only rigorously defined when $(Y^{2n-1}, \xi)$ is  $P^{2n-2} \times \mathbb{R}$ for some exact symplectic manifold $P$ 	 \cite{EES}; the latter is always defined.	
	A proof was sketched in \cite{BEE12} that these two invariants are quasi-isomorphic and for the results in the rest of this section, we will assume this.
	
	\begin{remark}
		Alternatively, let $CF(D^n, D^n; (W, \Lambda))$ denote the Floer cochains of the linking disk $D^n$ of $\Lambda$ in the partially wrapped Fukaya category of $W^{2n}$ stopped at $\Lambda$; a proof was sketched in \cite{EL} that this is quasi-isomorphic to the version of $CE(\Lambda)$ with coefficients in $C(\Omega S^{n-1})$, chains on the based loop space of $S^{n-1}$. By \cite{ganatra_generation}, 
	$CF(D^n, D^n; (W, \Lambda)) \otimes_{C(\Omega S^{n-1})} \mathbb{K}$ quasi-isomorphic to $CW(C,C)$ and so this invariant can be considered as a rigorous replacement for $CE(\Lambda)$; using this alternative invariant, all our results have complete proofs. 
\end{remark}
	
	Certain geometric properties of a Legendrian have algebraic consequences for its Chekanov-Eliashberg DGA. For example, an exact Lagrangian filling of $\Lambda$ induces an \textit{augmentation} of $CE(\Lambda)$, i.e. a differential graded algebra (DGA) map
	$CE(\Lambda) \rightarrow \mathbb{K}$, 	
	where the latter has the zero differential  and is concentrated in degree zero \cite{Ekholm_Honda_Kalman_cob}. 
	However, not all augmentations come from exact Lagrangian fillings \cite{etnyre_ng_survey} and furthermore, there are examples of Legendrians such that $CE(\Lambda)$ is not acyclic but admits no augmentations. 
	More generally, we can consider $n$-dimensional \textit{representations} of $CE(\Lambda)$, i.e.  DGA maps $CE(\Lambda) \rightarrow \mbox{Mat}(n, \mathbb{K})$.
	There are examples \cite{DR_Golovko_estimating, Sivek_maximal} of  Legendrians for which $CE(\Lambda)$ has a 2-dimensional representation but no augmentations. 
	This is a useful notion since	Dimitroglou-Rizell and Golovko \cite{DR_Golovko_estimating} showed that Legendrians with finite-dimensional representations have an Arnold-type lower bound on the number of Reeb chords. On the other hand, they showed that for each $n \ge 1$, there is a Legendrian $\Lambda \subset  (\mathbb{R}^{2n-1}, \xi_{std})$ such that $CE(\Lambda)$ is not acyclic but has no finite-dimensional representations (although any non-acyclic DGA has an infinite-dimensional ``representation" to its characteristic algebra \cite{ng_computable} ). 	
	These examples are obtained by spinning a particular 1-dimensional Legendrian studied by Sivek \cite{Sivek_maximal}, who proved that it has no finite-dimensional representations  by explicit calculation. We now show that such Legendrians occur generally. 
	
	Consider a Legendrian sphere $\Lambda$ in $(S^{n-1} \times S^n, \xi_{std}) = \partial(B^{2n}_{std}\cup H^{n-1})$, $n \ge 3$. In this case, there is a unique $\mathbb{Z}$-grading on $CE(\Lambda)$. Suppose furthermore that $\Lambda$ has algebraic intersection number one with $\{p\}\times S^n$ for some $p \in S^{n-1}$, i.e. 
	$[\Lambda] = \pm 1 \in H_{n-1}(S^{n-1}\times S^n; \mathbb{Z}) \cong \mathbb{Z}$ is primitive in homology. 
	This implies that $[\Lambda] =1 \in H_{n-1}(B^{2n}_{std}\cup H^{n-1}; \mathbb{Z}) \cong \mathbb{Z}$  and hence $\Lambda$ has no exact Lagrangian fillings in $B^{2n}_{std}\cup H^{n-1}$ for purely topological reasons. Hence there are no augmentations of $CE(\Lambda)$ that come from fillings. Our next result shows that $CE(\Lambda)$ has no augmentations at all and in fact, no finite-dimensional representations. 	 
	\begin{corollary}\label{cor: no_aug}
		If $\Lambda^{n-1} \subset (S^{n-1} \times S^n, \xi_{std}), n \ge 3,$ is Legendrian sphere that is primitive in homology, 		 then  $CE(\Lambda)$ has no finite-dimensional representations and no DGA maps to a commutative ring. 
	\end{corollary}
	If $\Lambda$ intersects $\{p\} \times S^n$ \textit{geometrically} once, then $\Lambda$ is a loose Legendrian \cite{Casals_Murphy_front}; see Section \ref{ssec: grad_traj}. 	
	In this case,  $CE(\Lambda)$ is acyclic and hence has no finite-representations for trivial reasons. 
	Corollary \ref{cor: no_aug} generalizes this to the case of  \textit{algebraic} intersection one. 
	The proof of Corollary \ref{cor: no_aug} goes roughly as follows. We first form the Weinstein domain $X^{2n}:=B^{2n}_{std} \cup H^{n-1} \cup H^n_{\Lambda}$; since	$[\Lambda] = \pm 1 \in H_{n-1}(S^{n-1}\times S^n; \mathbb{Z}) \cong \mathbb{Z}$, $X^{2n}$ is smoothly a ball. We then show that if there were an $n$-dimensional representation of $CE(\Lambda)$, then there would be a map $K_0(\mathcal{W}(X)) \rightarrow  K_0(Mat(n,  \mathbb{K}))$ and that  $[Mat(n,\mathbb{K})] \in K_0(Mat(n, \mathbb{K}))$ is in the image of this map. Since 
	$[Mat(n,\mathbb{K})] \in K_0(Mat(n, \mathbb{K}))$ is non-zero,  $K_0(\mathcal{W}(X))$ is also non-zero, a contradiction; we give the full proof in Section \ref{sec: proofs}.
	In fact, our proof shows that there are no DGA maps from $CE(\Lambda)$ to a ring $R$		for which  $[R] \in K_0(R)$ is non-zero. In particular, there are no DGA maps to rings satisfying the invariant basis number property or rank property. 
	We also note that our proof holds if the standard contact structure on $S^{n-1} \times S^n$ is replaced by another contact structure that has a Weinstein filling $W^{2n}$ with $H^*(W^{2n}; \mathbb{Z}) \cong H^*(B^{2n} \cup H^{n-1}; \mathbb{Z})$. 
	Although our proof of Corollary \ref{cor: no_aug} holds only for $n \ge 3$, the $n =2$ case for augmentations is also true  and was proven by Leverson \cite{Leverson_s1s2} using a different approach. When $\Lambda^{n-1}$ is not a sphere, our proof breaks down since we cannot attach a standard $n$-handle along $\Lambda$. However we can attach a generalized handle and so we expect the following result to hold: for any manifold $M^n$ with boundary $\Lambda^{n-1}$, the version of $CE(\Lambda)$ with coefficients in $C_*(\Omega M)$, chains on the based loop space of $M$, has no finite-dimensional representations; see \cite{EL}. If $M^n = D^n$,  $C_*(\Omega M) \cong \mathbb{K}$ and we recover Corollary \ref{cor: no_aug}.
	Finally, we note that a homological condition is necessary since there are Legendrian spheres, like the Legendrian unknot, in $(S^{n-1} \times S^n, \xi_{std})$ that have Lagrangian fillings in $B^{2n}_{std} \cup H^{n-1}$ and hence their Chekanov-Eliashberg DGA's have augmentations; of course, such Legendrians are zero in homology.
	
	Corollary \ref{cor: no_aug} can be used to study the $C^0$-topology of the space of Legendrians. As part of the h-principle for loose Legendrians, Murphy \cite{Murphy11} proved that any Legendrian can be $C^0$-approximated by a loose Legendrian. On the other hand, Dimitroglou-Rizell and Sullivan \cite{Rizell_sullivan_c0} recently used persistent homology to show that loose Legendrians cannot be $C^0$-approximated by certain non-loose Legendrians. More precisely, they showed that if $\Lambda_{loose} \subset    (\mathbb{R}^{2n-1}, \xi_{std})$ is a loose Legendrian and $\Lambda \subset (\mathbb{R}^{2n-1}, \xi_{std})$ can be Legendrian isotoped into the standard contact neighborhood $N(\Lambda_{loose})$ of $\Lambda_{loose}$ 	
	such that the map $i_*: H_{n-1}(\Lambda; \mathbb{Z}/2) \rightarrow
	H_{n-1}(N(\Lambda_{loose}); \mathbb{Z}/2) \cong  H_{n-1}(\Lambda_{loose}; \mathbb{Z}/2)$ is non-zero, then $CE(\Lambda)$ has no augmentations.
	Using Corollary \ref{cor: no_aug}, we give a different proof of a slightly different result.
	\begin{corollary}\label{cor: c0-close}
		If $\Lambda \subset (S^{2n-1}, \xi_{std}), n \ge 3,$ is a Legendrian sphere that can be  Legendrian isotoped into a standard contact neighborhood of the loose Legendrian unknot $\Lambda_{loose}$ and is a primitive homology class in $H_{n-1}(\Lambda_{loose}; \mathbb{Z})$, then 
		$CE(\Lambda)$ has no finite-dimensional representations or DGA maps to a commutative ring. 	
	\end{corollary}
	Hence the size of contact neighborhoods depends on  Legendrian isotopy class. In the proof of Corollary \ref{cor: c0-close}, the condition that $\Lambda$ is in a contact neighborhood of a loose Legendrian is used to show that a related Legendrian is  disjoint from the loose chart of another loose Legendrian. The homological condition is used to construct a Weinstein ball $X^{2n}$ and the fact that there is a disjoint loose chart in a particular Legendrian implies that 
	$D^b\mathcal{W}(X)$ is equivalent to $H^0(Tw(CE(\Lambda)))$; as in Corollary \ref{cor: no_aug}, this implies that $CE(\Lambda)$ has no finite-dimensional representations or DGA maps to a commutative ring. We note that some homology condition is  necessary
	since otherwise any Legendrian in $(S^{2n-1}, \xi_{std})$ can be isotoped into a neighborhood of any other Legendrian.

	Corollaries \ref{cor: no_aug}, \ref{cor: c0-close} place strong restrictions on the Chekanov-Eliashberg DGA's of certain Legendrians. Furthermore, if these Legendrians satisfy  stronger conditions, e.g. have \textit{geometric} intersection one with $\{p\} \times S^n$ instead of \textit{algebraic} intersection one, then they are loose, showing that there is not much room for interesting Legendrians. Nonetheless, we show that there are many examples of such Legendrians with non-trivial DGA's, essentially one for each exotic Weinstein ball; this shows that Corollaries \ref{cor: no_aug}, \ref{cor: c0-close} are  sharp.  
	\begin{corollary}\label{cor: exotic_leg_no_aug}
		For $n \ge 4$, there exist infinitely many different Legendrian spheres $\Lambda_k \subset (S^{n-1} \times S^n, \xi_{std})$ for which $CE(\Lambda_k)$ is not acyclic but has no  finite-dimensional representations. 
		The same holds for $(S^{2n-1}, \xi_{std}), n \ge 4$. Furthermore, these Legendrians are $C^0$-close to loose Legendrians $\Lambda_{loose}$ and are primitive in $H_{n-1}(\Lambda_{loose}; \mathbb{Z})$. 
	\end{corollary} 
	The restriction $n \ge 4$ comes from the fact that we currently have examples of exotic Weinstein balls only in such dimensions \cite{MM}.
	The Legendrians $\Lambda_k$ are distinguished by the Hochschild homology of $CE(\Lambda_k)$, which is related to invariants of these Weinstein balls. 	
\\

	Now we give an outline of the rest of the paper. In Section \ref{sec: background}, we provide some background material on Weinstein domains, loose Legendrians, and handle-slides. In Section \ref{sec: proofs}, we give proofs of the results stated in the Introduction.

	\section*{Acknowledgements} 
	We thank Mohammed Abouzaid, Roger Casals, Emmy Murphy, Kyler Siegel, Semon Rezchikov, and Ivan Smith for many helpful discussions. This work was partially supported by an NSF postdoc fellowship. 
	
	\section{Background}\label{sec: background}
	
	In this section, we present some background material, including necessary definitions and theorems that were assumed in the Introduction. 
	
	\subsection{Liouville and Weinstein domains}\label{ssec: domains}

	\subsubsection{Definitions}
	A \textit{Liouville domain} is a pair $(W^{2n}, \lambda)$ such that \begin{itemize}
		\item $W^{2n}$ is a compact manifold with boundary
		\item $d\lambda$ is a symplectic form on $W$ 
		\item the Liouville field $X_\lambda$, defined by $i_{X}d\lambda = \lambda$, is outward transverse along $\partial W$.
	\end{itemize}
	A \textit{Weinstein domain} is a triple $(W^{2n}, \lambda, \phi)$ such that
	\begin{itemize}
		\item $(W, \lambda)$ is a Liouville domain
		\item $\phi: W \rightarrow \mathbb{R}$ is a Morse function with maximal level set $\partial W$ 
		\item $X_\lambda$ is a gradient-like vector field for $\phi$.
	\end{itemize}
	Liouville and Weinstein \text{cobordisms} are defined similarly.  
	
	Since $W$ is compact and $\phi$ is a Morse function with maximal level set $\partial W$, $\phi$ has finitely many critical points. We will call $\phi$ a Weinstein Morse function. Note that for any regular value $c$, $W^c = \{\phi \le c\}$ is also a Weinstein domain and is called a Weinstein \textit{subdomain}.
	
	If $\Sigma^{2n-1} \subset (W^{2n}, \lambda)$ is a hypersurface such that $X_\lambda$ is transverse to $\Sigma$, then $\ker (\lambda|_\Sigma)$ is a contact structure on $\Sigma$. In the Weinstein case, a regular level set $Y^c = \phi^{-1}(c)$ of $\phi$ is such a hypersurface and so $(Y^c, \lambda|_{Y^c})$ is a contact manifold. In particular, the boundary $\partial W$ of Liouville or Weinstein domain $W$ has a natural contact structure given by $\xi = \ker (\lambda|_{\partial W})$. The \textit{completion} $\widehat{W}$ of $W$ is the non-compact, exact symplectic manifold obtained by attaching the symplecticization 
	$(\partial W \times [0, \infty), d(e^t \lambda|_{\partial W}))$ of $(\partial W, \xi)$ to $W$. Whenever we speak of the symplectomorphism type of a Weinstein domain, we will mean the symplectomorphism type of its completion.

	\subsubsection{Weinstein handle attachment}
	
	A Weinstein structure yields a special handle-body decomposition for $W$. First, recall that $\lambda$ vanishes on the $X_\lambda$-stable disc $D_p$ of a critical point $p$; see \cite{CE12}. In particular, $D_p$ is isotropic with respect to $d\lambda$ and so all critical points of $\phi$ have index less than or equal to $n$. If all critical points of $\phi$ have index \textit{strictly less} than $n$, then the Weinstein domain is \textit{subcritical}.  
	
	Since $\lambda$ vanishes on $D_p$, then  $\Lambda_p: = D_p \cap Y^c \subset (Y^c, \lambda|_{Y^c})$ is an isotropic sphere, where $c = \phi(p) - \epsilon$ for sufficiently small $\epsilon$. Furthermore, $\Lambda_p$ comes with a parametrization and framing, i.e. a trivialization of its normal bundle. Note that a framing of $\Lambda_p$ is equivalent to the framing of the conformal symplectic normal bundle of $\Lambda_p$; see \cite{Gbook}. Hence parametrized Legendrians come with a canonical framing.
	
	Suppose that $c_1 < c_2$ are regular values of $\phi$ and 
	$W^{c_2} \setminus W^{c_1}$ contains a unique critical point $p$ of $\phi$. Then $W^{c_2} \backslash W^{c_1}$ is an elementary Weinstein cobordism between $Y^{c_1}$ and $Y^{c_2}$ and the symplectomorphism type of $W^{c_2}$ is determined by the symplectomorphism type of $W^{c_1}$ along with the framed isotopy class of the isotropic sphere $\Lambda_p \subset Y^{c_1}$.
	If  $\phi$ is an arbitrary Weinstein Morse function on $W$ with distinct critical values, then $W$ can be viewed as the concatenation of such elementary Weinstein cobordisms. 
	
	On the other hand, one can explicitly construct such elementary cobordisms and use them to modify Liouville domains. Given a Liouville domain $X$ and a framed isotropic sphere $\Lambda$ in its contact boundary $Y= \partial X$, we can attach an elementary Weinstein cobordism with critical point $p$ and $\Lambda_p = \Lambda$ to $X$ and obtain a new Liouville domain that we denote by $X_\Lambda$ or $X \cup H^k_\Lambda$, where $k = \mbox{ind } p = \dim \Lambda + 1 $.
	This operation is called \textit{Weinstein handle attachment} and $\Lambda$ is called the \textit{attaching sphere} of the Weinstein handle.
	If $X$ is Weinstein, then so is $X_\Lambda$. 
	If the dimension of $\Lambda \subset Y^{2n-1}$ is less than $n-1$, the handle attachment operation and $\Lambda$ itself are all called \textit{subcritical}.
	So any  (subcritical) Weinstein domain can be obtained by  attaching (subcritical) Weinstein handles to the standard Weinstein structure on $B^{2n}$. 
	
	The corresponding modification of contact manifolds by Weinstein handle attachment is called \textit{contact surgery}. If $\Lambda \subset (Y, \xi)$ is a framed isotropic sphere, then there exists an elementary Weinstein cobordism $W$ with $\partial_- W  =  (Y, \xi)$ and attaching sphere $\Lambda$. Then we say $\partial_+W$ is the result of contact surgery on $\Lambda$ and denote this by $Y_\Lambda$ or $Y \cup H^k_\Lambda$. 
	In particular, the contact boundary of any (subcritical) Weinstein domain can be obtained by doing (subcritical) contact surgery to $(S^{2n-1}, \xi_{std}) = \partial B^{2n}$.

	\subsubsection{Weinstein homotopies}\label{sssec: handleslides}
	
	The natural notion of equivalence between Weinstein structures
	$(W, \lambda_0, \phi_0),$ $(W, \lambda_1, \phi_1)$ on a fixed manifold $W$ is a 
	\textit{Weinstein homotopy}, i.e. a 1-parameter family
	of Weinstein structures $(W, \lambda_t, \phi_t)$, 
	$t\in [0, 1],$ connecting them, where $\phi_t$ is allowed to have birth-death critical points. Weinstein homotopic domains have exact  symplectomorphic completions \cite{CE12}. 
	
	We will prove our main result Theorem \ref{thm: flexdomainhandle} by starting with an arbitrary Weinstein domain and then applying a special Weinstein homotopy. As in the smooth setting, Weinstein homotopies consist of three elementary moves:  doing an isotopy of the attaching spheres through isotropic submanifolds, moving critical points that are not connected by gradient trajectories past each other, and sliding handles of the same index over each other. The only difference between the Weinstein and smooth setting is the first move: in the Weinstein case, the isotopies of attaching spheres must be through isotropics instead of arbitrary embedded spheres. Since subcritical handles satisfy an h-principle \cite{CE12}, Weinstein domains are essentially characterized by their index $n$ handles, in particular the Legendrian attaching spheres of these critical handles. Therefore, it suffices to see how these moves affect Legendrians. 
	
	The first move implies that if $\Lambda_1, \Lambda_2$ are isotopic Legendrians in $\partial W$, then $W\cup H_{\Lambda_1}^n$ and $W\cup H_{\Lambda_2}^n$ are Weinstein homotopic. 
	The second move implies that if $\Lambda_1, \Lambda_2$ are disjoint Legendrians in $\partial W$ (which is true by dimension reasons if they are in general position), then $(W\cup H_{\Lambda_1}^n) \cup H_{\Lambda_2}^n$ and 
	$(W\cup H_{\Lambda_2}^n) \cup H_{\Lambda_1}^n$ are Weinstein homotopic. In particular, we can write the resulting Weinstein domain as $W\cup H_{\Lambda_1}^n \cup H_{\Lambda_2}^n$ without any parentheses and it will be well-defined up to Weinstein-homotopy.
	
	We now discuss the last move, the handle-slide, which will be 
	the most important for us. We will study Legendrians via their front projection. If $\Lambda \subset (\mathbb{R}^{2n+1}, \xi_{std}) = \mathbb{R}^n \times \mathbb{R}^n \times \mathbb{R}^1$, the front projection of $\Lambda$ is the image of $\Lambda$ in 
	$\mathbb{R}^{n+1}$ under the projection to the first $\mathbb{R}^n$ and $\mathbb{R}^1$ components. Handles-slides were described in terms of front projections by Casals and Murphy \cite{Casals_Murphy_front}. 
	\begin{proposition}\label{prop: handleslide}[Proposition 2.4 of \cite{Casals_Murphy_front}]
		Let $(Y, \xi)$ be a contact manifold and $ \Lambda, \Sigma \subset (Y, \xi)$ be 
		two disjoint Legendrian submanifolds such that $\Lambda$ is a sphere. Suppose there exists a Darboux chart $U$ where the front projections of $\Sigma, \Lambda$ look as in the left-hand-side of Figure 1. Then for sufficiently small $\epsilon>0$, the Legendrians $\Sigma$ and $h_{\Lambda, \epsilon}(\Sigma)$ presented in Figure \ref{fig: handslide_model} are Legendrian isotopic in the surgered contact manifold $Y_{\Lambda}$.
		\begin{figure}
			\centering
			\includegraphics[scale=0.2]{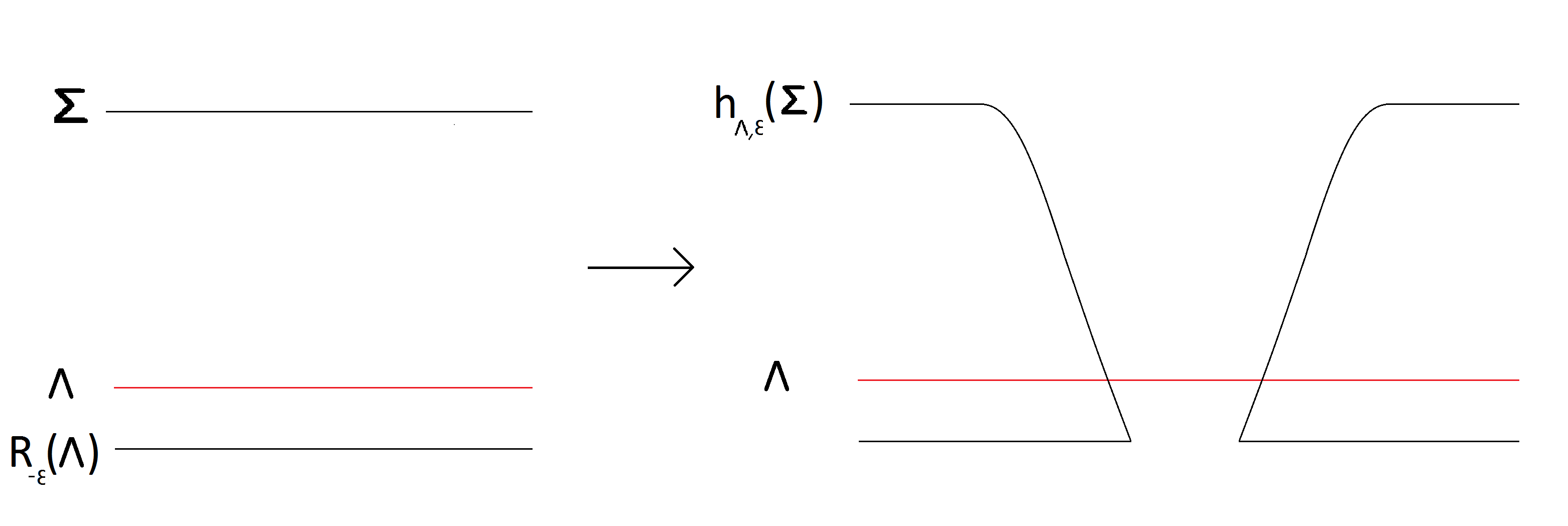}
			\caption{Front projection of handle-slide $h_{\Lambda}(\Sigma)$ of $\Sigma$ over $\Lambda$.}
			\label{fig: handslide_model}
		\end{figure}
	\end{proposition}
	Here $R_{-\epsilon}(\Lambda)$ is the image of $\Lambda$ under the negative time $\epsilon$ Reeb flow. We also note that the Legendrians in Figure \ref{fig: handslide_model} are extended by spherical symmetry out of the page. Furthermore, we note that the Darboux chart must have sufficient size so that front projections depicted in Figure \ref{fig: handslide_model} make sense; in particular, the size of the chart in the $y_i$ direction must be at least as big as the slope of the front projection of $h_{\Lambda, \epsilon}(\Lambda)$. 
	For us, the key implication of Proposition \ref{prop: handleslide} is that $W \cup H_\Lambda^n \cup H^n_\Sigma$  is Weinstein homotopic to 
	$W \cup H^n_\Lambda \cup  H_{h_\Sigma(\Lambda)}^n$ (and also to 
	$W \cup   H_{h_\Sigma(\Lambda)}^n \cup H^n_\Lambda$ by the above discussion).
	
	\begin{remark}\label{rem: link_handleslide}
		Proposition \ref{prop: handleslide} also holds if $\Sigma = \Sigma_1 \coprod \cdots \coprod \Sigma_k$ is a Legendrian link with several components. We inductively construct the new handle-slid link and show that it is isotopic to $\Sigma$ in $Y_\Lambda$. 
		We first take sufficiently small $\epsilon_1> 0$  so that $\Sigma$ is disjoint from an $\epsilon_1$-neighborhood of $\Lambda$ in $J^1(\Lambda) \subset Y$. We also take $U_1$ so that $\Sigma_1\cap U_1, \Lambda\cap U_1$ look as in the left-hand-side of Figure \ref{fig: handslide_model} and $\Sigma_i \cap U_1  = \emptyset$ for $i \ge 2$. Then we can handle-slide  $\Sigma_1$ over $\Lambda$ via $U_1$ and the resulting Legendrian $h_{\Lambda, \epsilon_1}(\Sigma_1)$ is isotopic to $\Sigma_1$ in $Y_{\Lambda}$ by Proposition \ref{prop: handleslide}. In fact, something stronger holds. The isotopy in Proposition \ref{prop: handleslide} is local since it is obtained by pushing a small disk of $\Sigma_1$ (starting from the chart $U_1$) past the belt sphere of $\Lambda$ in $Y_{\Lambda}$. Therefore since $\Sigma_2, \cdots, \Sigma_k$ are disjoint from an $\epsilon_1$-neighborhood of $\Lambda$ in $Y$ and the chart $U_1$,  the handle-slid Legendrian $h_{\Lambda, \epsilon_1}(\Sigma_1)$ is isotopic to $\Sigma_1$ in $Y_{\Lambda} \backslash (\Sigma_2 \coprod \cdots \coprod \Sigma_k)$, where we view $\Sigma_2, \cdots, \Sigma_k$ as Legendrians of $Y_{\Lambda}$. Hence the link 
		$h_{\Lambda, \epsilon_1}(\Sigma_1) \coprod \Sigma_2 \coprod \cdots \coprod \Sigma_k$ is isotopic to $\Sigma_1 \coprod \Sigma_2 \coprod \cdots \coprod \Sigma_k$ in $Y_{\Lambda}$. Now we build the rest of the handle-slid link by induction and show that it is isotopic to the original link $\Sigma$ at each stage. Namely, suppose we have constructed the $i$th link $h_{i}(\Sigma):=h_{\Lambda, \epsilon_1} (\Sigma_1) \coprod \cdots \coprod h_{\Lambda, \epsilon_i} (\Sigma_i) \coprod \Sigma_{i+1} \coprod \cdots \coprod \Sigma_k$ and proved that it is isotopic to $h_{i-1}(\Sigma)$ in $Y_{\Lambda}$. Next we construct $h_{i+1}(\Sigma):= h_{\Lambda, \epsilon_1} (\Sigma_1) \coprod \cdots \coprod h_{\Lambda, \epsilon_i} (\Sigma_i) \coprod h_{\Lambda, \epsilon_{i+1}} (\Sigma_{i+1}) \coprod \Sigma_{i+2} \coprod \cdots \coprod \Sigma_k$ by taking sufficiently small $\epsilon_{i+1} < \epsilon_j$ for all $j\le i$ 
		and a chart $U_{i+1}$ disjoint from $h_{i}(\Sigma) \backslash \Sigma_{i+1}$ such that $\Sigma_{i+1}, \Lambda$ appear in $U_{i+1}$ as in Figure \ref{fig: handslide_model}. As explained above, the new link 
		$h_{i+1}(\Sigma)$ is Legendrian isotopic to the previous link $h_{i}(\Sigma)$ in $Y_{\Lambda}$ since $h_{i}(\Sigma) \backslash \Sigma_{i+1}$ is disjoint from $U_{i+1}$ and $h_i(\Sigma)$ is disjoint from an $\epsilon_{i+1}$-neighborhood of $\Lambda$ (since the Legendrians in $h_{i}(\Sigma)$ are at most $\epsilon_{i}$-close to $\Lambda$), which proves the inductive $i+1$ case. For $i = k$, we get the desired Legendrian  $h_k(\Sigma)$ which is isotopic to $\Sigma$ in $Y_{\Lambda}$ by induction. 
		This implies that $W \cup H^n_{\Lambda} \cup H^n_{\Sigma_1} \cup \cdots \cup  H^n_{\Sigma_k}$ is Weinstein homotopic to 
		$W \cup H^n_{\Lambda} \cup H^n_{h_{\Lambda}(\Sigma_1)} \cup \cdots \cup  H^n_{h_{\Lambda}(\Sigma_k)}$, a fact that we will use repeatedly later. 
	\end{remark}

	We also note that the handle-slide depend on more than just the data of $\Sigma$ and $\Lambda$. The resulting Legendrian depend crucially on the choice of  chart $U$ where $\Lambda, \Sigma$ appear as in the left-hand-side of Figure \ref{fig: handslide_model}. In particular, different chart choices can result in Legendrians $h_{\Lambda, \epsilon}(\Sigma)$ that are not Legendrian isotopic in $Y$ while still being Legendrian isotopic in $Y_\Lambda$. 
	
	\begin{examples}
		We start with a Legendrian link consisting of two linked unknots in $(\mathbb{R}^{2n-1}, \xi_{std})$, with one Legendrian the Reeb push-off of the other Legendrian; see 
		Figure \ref{fig: handslide_examples}. The two light-blue boxes are the Darboux charts used in the handleslides. 
		In the top row, the handle-slide produces a linked pair of Legendrian unknots (which can be seen by doing a Legendrian Reidermeister move), i.e $h^{top}_{\Lambda_{unknot}}(\Lambda_{unknot}) = \Lambda_{unknot}$. 
		In the bottom row, the handle-slide results in a link where one of the Legendrians is loose, i.e. $h^{bottom}_{\Lambda_{unknot}}(\Lambda_{unknot}) = \Lambda_{loose}$.  
		The dark blue box is the loose chart of this Legendrian; see Section \ref{subsection: loose} for definition. Since the Legendrian unknot is not loose, the handle-slid Legendrians $h^{top}_{\Lambda_{unknot}}(\Lambda_{unknot}), h^{bottom}_{\Lambda_{unknot}}(\Lambda_{unknot})$ are not isotopic in the original contact manifold $(\mathbb{R}^{2n-1}, \xi_{std})$. Of course, these Legendrians
		are both isotopic in the surgered manifold $Y_{\Lambda_{unknot}}$ since they are both isotopic to the push-off of the attaching sphere there, i.e the image of $\Lambda_{unknot}$ in $Y_{\Lambda_{unknot}}$. 
	\end{examples}

	\begin{figure}
		\hspace*{-01.3cm}  
		\includegraphics[scale=0.19]{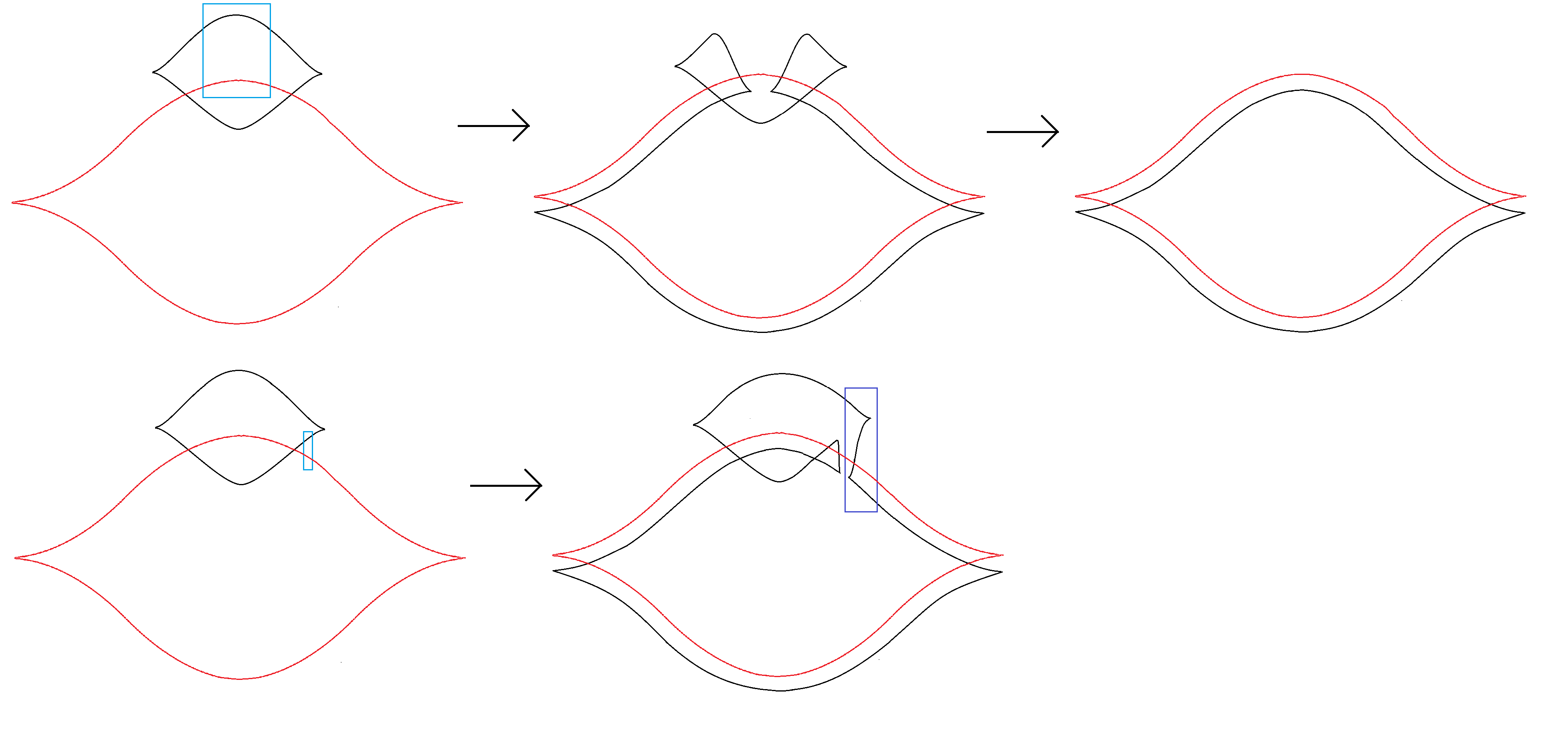}
		\caption{Handle-slides using different charts result in non-isotopic Legendrians.}              
		\label{fig: handslide_examples}
	\end{figure}

	\subsection{Loose Legendrians and flexible Weinstein domains}\label{subsection: loose}
	
	There exist many Legendrians with rich symplectic topology invisible from the point of view of algebraic topology. On the other hand, Murphy  \cite{Murphy11} showed that exists a certain class of \textit{loose} Legendrians which satisfy a h-principle and whose symplectic topology is governed by their underlying algebraic topology.
	There are several equivalent criteria for a Legendrian to be loose, all of which depend the existence of a certain local model inside this Legendrian. We will use the following local model from Section 2.1 of \cite{CMP}. 
	Let $B^3 \subset (\mathbb{R}^3, \xi_{std} = \ker \alpha_{std})$ be a unit ball and let $\Lambda_0$ be the 1-dimensional Legendrian whose front projection is shown in Figure \ref{fig: stabilizationgeo}. Let $Q^{n-2}, n \ge 3,$ be a closed manifold
	and $U$ a neighborhood of the zero-section $Q \subset T^*Q$. Then 
	$\Lambda_0 \times Q \subset (B^3 \times U, \mbox{ker}(\alpha_{std} + \lambda_{std}))$ is a Legendrian submanifold. 
	This Legendrian is the \textit{stabilization} over $Q$ of 
	the Legendrian $\{y = z = 0 \} \times Q \subset (B^3 \times U, \mbox{ker}(\alpha_{std} + \lambda_{std}))$.
	\begin{definition}\label{def: loose}
		A Legendrian $\Lambda^{n-1} \subset (Y^{2n-1}, \xi), n \ge 3,$ is \textit{loose} if there is a neighborhood $V\subset (Y, \xi)$ of $\Lambda$ such that $(V, V\cap \Lambda)$ is contactomorphic to $(B^3 \times U, \Lambda_0 \times Q)$.
	\end{definition}
	\begin{remark}\label{rem: looseleg}
		If $f: (U^{2n-1}, \xi_1) \rightarrow (V^{2n-1}, \xi_2)$ is an 
		equidimensional contact embedding and $\Lambda \subset (U, \xi_1)$ is loose, then $f(\Lambda) \subset (V, \xi_2)$ is also loose.
	\end{remark}

	\begin{figure}
		\centering
		\includegraphics[scale=0.15]{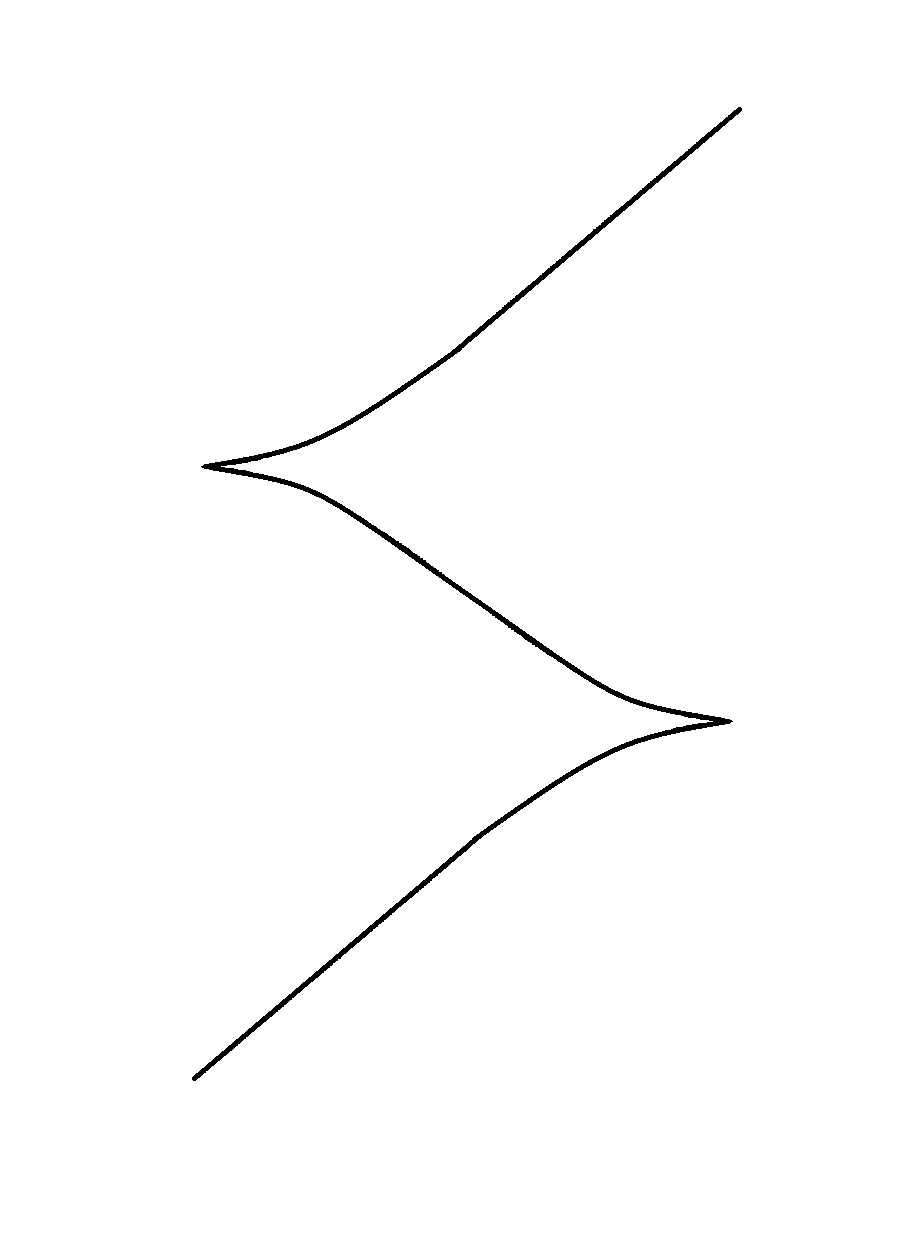}
		\caption{Front projection of $\Lambda_0$.}
		\label{fig: stabilizationgeo}
	\end{figure}

	A \textit{formal Legendrian embedding} is an embedding $f: \Lambda \rightarrow (Y, \xi)$ together with a homotopy of bundle monomorphisms $F_s: T\Lambda \rightarrow TY$ covering $f$ for all $s$ such that $F_0 = df$ and $F_1(T\Lambda)$ is a Lagrangian subspace of $\xi$ with its conformal symplectic structure.  A \textit{formal Legendrian isotopy} is an isotopy through formal Legendrian embeddings. Using these notions, we can state Murphy's h-principle \cite{Murphy11}, which has an existence and uniqueness part:
	\begin{itemize}
		\item any formal Legendrian of dimension at least two is formally Legendrian isotopic to a loose Legendrian
		\item any two loose Legendrians that are formally Legendrian isotopic are genuinely Legendrian isotopic. 
	\end{itemize}

	We now define a class of Weinstein domains introduced in \cite{CE12} that are constructed by iteratively attaching Weinstein handles along loose Legendrians.
	
	\begin{definition}\label{def: weinstein_flexible}
		A Weinstein domain $(W^{2n}, \lambda, \phi), n \ge 3,$ is  \textit{flexible} if there exist regular values $c_1, \cdots, c_{k}$ of $\phi$ such that $c_1 < \min \phi < c_2 < \cdots < c_{k-1} < \max \phi < c_{k}$ and for all $i = 1, \cdots, k-1$, $\{c_i \le \phi \le c_{i+1} \}$ is a Weinstein cobordism with a single critical point $p$ whose the attaching sphere $\Lambda_p$ is either subcritical or a loose Legendrian in $(Y^{c_i}, \lambda|_{Y^{c_i}})$.  
	\end{definition} 
	
	Flexible Weinstein \textit{cobordisms} are defined similarly. Also, Weinstein handle attachment or contact surgery is called flexible if the attaching Legendrian is loose. 
	So any flexible Weinstein domain can be constructed by iteratively attaching subcritical or flexible handles to $(B^{2n}, \omega_{std})$.  A Weinstein domain that is Weinstein homotopic to a Weinstein domain  satisfying Definition \ref{def: weinstein_flexible} will also be called flexible. Finally, we note that subcritical domains are automatically flexible.
	
	Our definition of flexible Weinstein domains is a bit different from the original definition in \cite{CE12}, where several critical points are allowed in  $\{c_i \le \phi \le c_{i+1} \}$. 
	There are no gradient trajectories between these critical points and their attaching spheres form a loose \textit{link} in $(Y^{c_i}, \lambda|_{Y^{c_i}})$, i.e each Legendrian is loose in the complement of the others. 
	These two definitions are the same up to Weinstein homotopy. Indeed if we have an ordered collection of Legendrians such that each one is loose in the complement of the previous ones, then we can use the loose Legendrian h-principle to move each Legendrian away from the loose charts of the previous ones so that all Legendrians are loose in the complement of each other.

	Since they are built using loose Legendrians, which satisfy an h-principle, flexible Weinstein domains also satisfy an h-principle as proven by Cieliebak and Eliashberg \cite{CE12}.
	Again, the h-principle has an existence and uniqueness part:
	\begin{itemize}
		\item any almost Weinstein domain of dimension at least six admits a flexible Weinstein structure in the same almost symplectic class
		\item any two flexible Weinstein domains that are almost symplectomorphic are Weinstein homotopic (and hence have exact symplectomorphic completions and contactomorphic boundaries).
	\end{itemize}

	\section{Proofs of Main Results}\label{sec: proofs}
	In this section, we prove the results described in the Introduction. 
	We first prove a simpler version of Theorem \ref{thm: flexible_subdomain} without as much control on the topology of the flexible subdomain. 
	
	\begin{theorem}\label{thm: flexdomainhandle}
		Any Weinstein domain $W^{2n}, n\ge 3,$ can be Weinstein homotoped to a Weinstein domain $V_{flex}^{2n} \cup H^n$ obtained by attaching a single $n$-handle to a flexible Weinstein domain $V_{flex}^{2n}$. 
	\end{theorem}
	\begin{remark}
		Theorem \ref{thm: flexdomainhandle} also holds for Weinstein cobordisms.
	\end{remark}
	\begin{proof}[Proof of Theorem \ref{thm: flexdomainhandle}]
		
		Let $W^{2n} = (W^{2n}, \lambda, \phi), n \ge 3,$ be a Weinstein domain. By Lemma 12.20 of \cite{CE12}, we can Weinstein homotope $W$ so that $\phi$ is self-indexing, i.e. if $p$ is a critical point of index $k$, then $\phi(p) = k$. In particular, we can assume that $W$ is the result of attaching $k$ index $n$ handles to a subcritical Weinstein domain $W_{sub}$ along disjoint Legendrians $\Lambda_1,  \cdots, \Lambda_k$. 
		
		If $k = 0$, then $W = W_{sub}  = W_{sub} \cup H^{n-1} \cup H^n$, where $H^{n-1}, H^n$ are two cancelling handles of index $n-1$ and $n$; the domain $W_{sub}\cup H^{n-1}$ is subcritical and hence flexible. If $k = 1$, then $W= W_{sub} \cup H^n_{\Lambda_1}$; again $W_{sub}$ is subcritical and hence flexible. Therefore we can assume $W = W_{sub} \cup H^n_{\Lambda_1}\cup   \cdots \cup H^n_{\Lambda_k}$ for some $k \ge 2$. 
		
		The key step is to handle-slide $H_{\Lambda_2}, \cdots, H_{\Lambda_k}$ over $H_{\Lambda_1}$. We will do this by induction. More precisely, we will prove that for every $j$ with $2 \le j \le k$, $W$ is Weinstein homotopic to $W_{sub} \cup H^n_{\Lambda_1'} \cup  \cdots \cup H^n_{\Lambda_k'}$ for some Legendrian link $\coprod_{i=1}^k \Lambda_i'$ such that $\coprod_{i=2}^j \Lambda_i'$ is a loose link in $\partial W_{sub}$. Then the case $j = k$ completes the proof since then $W$ is Weinstein homotopic to the flexible domain $W_{sub} \cup H^n_{\Lambda_2'} \cup \cdots \cup H^n_{\Lambda_k'}$ with the single handle $H^n_{\Lambda_1'}$ attached. The proof shows that we can assume that $\Lambda_1$ actually stays fixed throughout. 
		
		We first prove the base case $j = 2$. We begin by 
		modifying $\Lambda_1, \Lambda_2$ by Legendrian isotopies that move only a small neighbhorhood of a single point, i.e. the resulting Legendrians are the Legendrian connected sum of $\Lambda_1, \Lambda_2$ with certain Legendrian unknots. More precisely, let $U_2$ be a Darboux ball in the contact manifold $\partial W_{sub}$ that is disjoint from $\Lambda_1 \cup \cdots \cup \Lambda_k$.  Let $S_2$ be a Legendrian unknot in $U_2$ and let $T_2$ be a negative Reeb push-off of $S_2$ also contained in $U_2$ so that $S_2, T_2$ are symplectically unlinked. We apply a Legendrian ``Reidemeister move" to $S_2$ so that it appears as in Figure \ref{fig: twounknots}; this move is a Legendrian isotopy which is contained in $U_2$ and the resulting Legendrian, which we also call $S_2$, is still symplectically unlinked with $T_2$. For one-dimensional Legendrians, this isotopy is the first Reidemeister move and in higher dimenions (as in our situation) it  results in a spun version of this Reidemeister move, although the isotopy is not obtained by spinning the one-dimensional isotopy; see \cite{Casals_Murphy_front} for details on this isotopy.

		Now we choose isotropic arcs $\gamma_1, \gamma_2$ connecting $\Lambda_1$ to $T_2$ and $\Lambda_2$ to $S_2$ respectively. Since these arcs are subcritical, we can assume that they are disjoint; furthermore, we can assume that $\gamma_1$ is disjoint from $\Lambda_i, i \ne 1$ and $\gamma_2$ is disjoint from $\Lambda_i, i \ne 2$. 
		We can also ensure that they intersect $U_2$ as depicted in the left-hand-side of Figure \ref{fig: twounknots}. Let $\Lambda_1' := \Lambda_1 \sharp T_2$ be the Legendrian connected sum of $\Lambda_1$ and $T_2$ along $\gamma_1$; see \cite{Riz} for details about the connected sum operation. Similarly, let $\Lambda_2' := \Lambda_2 \sharp S_2$ be the Legendrian connected sum of $\Lambda_2$ and $S_2$ along $\gamma_2$.  By choice of $\gamma_1, \gamma_2$, the Legendrians $\Lambda_1' \cap U_2, \Lambda_2' \cap U_2$ look as in right-hand-side of Figure \ref{fig: twounknots}. 
		Since $U_2$ is disjoint from $\Lambda_1$ and $T_2$ is a Legendrian unknot in $U_2$, $\Lambda_1'$ is isotopic to $\Lambda_1$; we pull the unknot $T_2$ to $\Lambda_1$ using the isotropic arc $\gamma_1$. 
		Similarly, $\Lambda_2'$ is Legendrian isotopic to $\Lambda_2$. In fact,  the whole Legendrian link $\Lambda_1' \coprod \Lambda_2' \coprod \Lambda_3 \coprod \cdots \coprod \Lambda_k$ is Legendrian isotopic to the link  $\Lambda_1 \coprod \Lambda_2 \coprod \Lambda_3 \coprod \cdots \coprod \Lambda_k$ because 
		$\gamma_1, \gamma_2$ are disjoint from  $\Lambda_3, \cdots, \Lambda_k$ and $S_2, T_2$ are symplectically unlinked in $U_2$. 
		
		\begin{figure}
			\centering
			\includegraphics[scale=0.2]{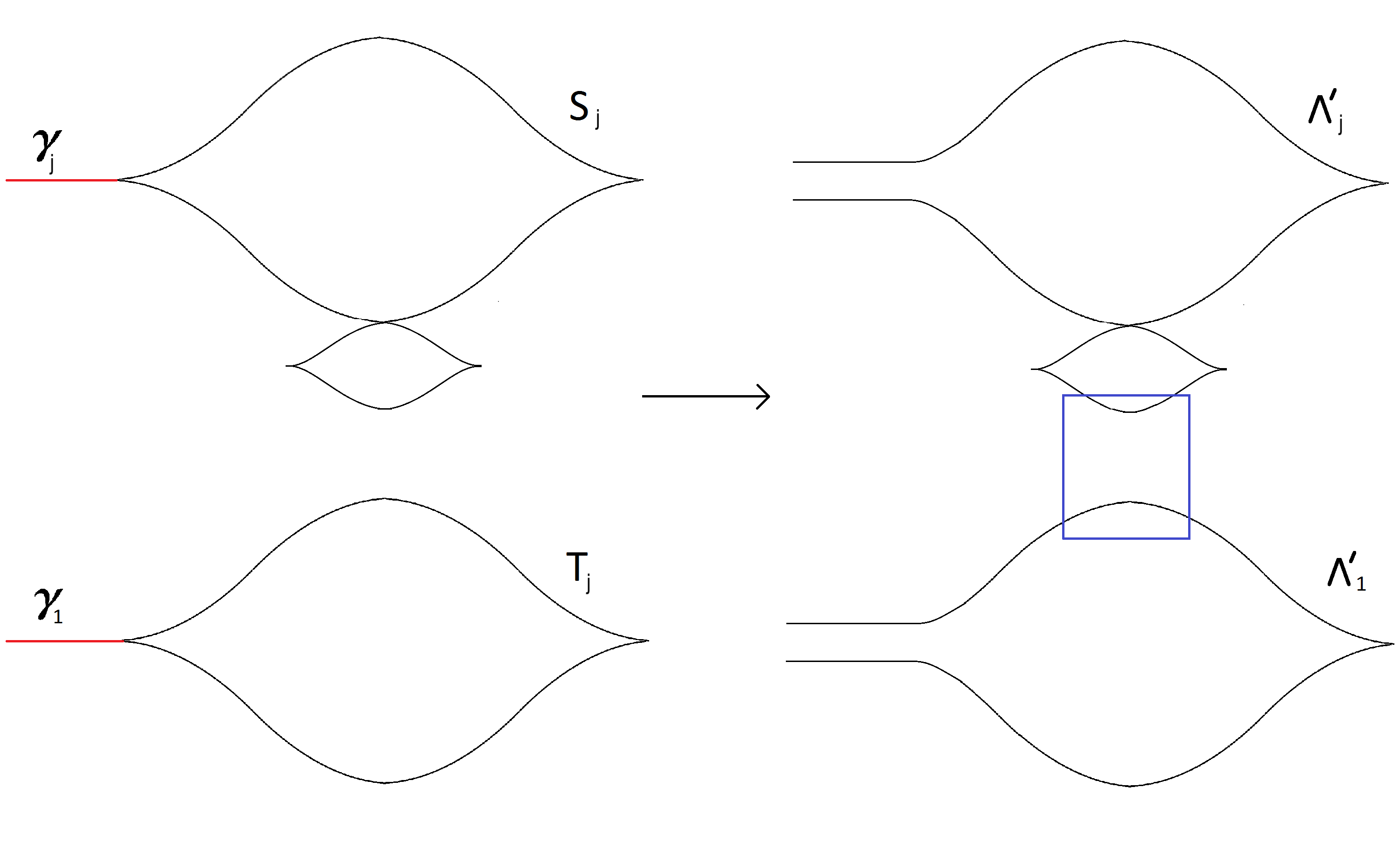}
			\caption{Left-hand figure: front projections of $S_j, T_j$, and isotropic arcs $\gamma_{1}, \gamma_{j}$ (in red) in $U_j$; right-hand figure: front projections of $\Lambda_j'$, the connected sum of $\Lambda_j, \Lambda_1$ and $S_j,T_j$  along $\gamma_j, \gamma_{1}$  respectively; the blue box is the chart we will use to handle-slide $\Lambda_j'$ over $\Lambda_1'$.}
			\label{fig: twounknots}
		\end{figure}

		Now we handle-slide $\Lambda_2'$ over $\Lambda_1'$. 
		We first take sufficiently small  $\epsilon_2 > 0$ so that
		an $\epsilon_2$-neighborhood of $\Lambda_1'$ is disjoint from all other Legendrians. 
		The ball $U_2$ contains a smaller chart $V_2$ where $\Lambda_1', \Lambda_2'$ look as in Figure \ref{fig: handslide_model}; see the blue box in the right-hand side of Figure \ref{fig: twounknots}. So we can use this chart to handle-slide $\Lambda_2'$ over $\Lambda_1'$ and produce $h_{\Lambda_1', \epsilon_2}(\Lambda_2')$; see the Legendrian in black in the right-hand-side of Figure \ref{fig: handleslide}. Then 
		$h_{\Lambda_1', \epsilon_2}(\Lambda_2')$ is isotopic to the Legendrian $\Lambda_2'$ in $\partial(W_{sub} \cup H^n_{\Lambda_1'})$; in fact, the whole link 
		$h_{\Lambda_1', \epsilon_2}(\Lambda_2') \coprod \Lambda_3 \coprod \cdots \coprod \Lambda_k$ is Legendrian isotopic to 
		$\Lambda_2' \coprod \Lambda_3 \coprod \cdots \coprod \Lambda_k$ in $\partial(W_{sub} \cup H^n_{\Lambda_1'})$ as explained in Remark \ref{rem: link_handleslide}.
		In particular,
		$W_{sub} \cup H^n_{\Lambda_1'} \cup H^n_{h_{\Lambda_1', \epsilon_2}(\Lambda_2')} \cup
		H^n_{\Lambda_3} \cup \cdots \cup H^n_{\Lambda_k}$ is Weinstein homotopic to 
		$W_{sub} \cup H^n_{\Lambda_1'} \cup H^n_{\Lambda_2'} \cup H^n_{\Lambda_3} \cup \cdots \cup H^n_{\Lambda_k}$ and hence to $W$. 
		Finally, we note that the size requirement of the Darboux chart for the handle-slide is satisfied in our situation. We can take the bottom branch of $S_2$ and the top branch of $T_2$ to be arbitrarily close so that the slope of the front projection of the handle-slid Legendrian is arbitrarily small; hence the $y_i$ coordinate of the chart can be arbitrarily small for our handle-slide. 
		
		We observe that $h_{\Lambda_1'}(\Lambda_2')$ is loose in $\partial W_{sub}$. The blue box in Figure \ref{fig: handleslide} is the loose chart of  $h_{\Lambda_1', \epsilon_2}(\Lambda_2')$ in $U_2$. Recall that we have spherical symmetry in the handle-slide region so it is loose with $Q^{n-2} = S^{n-2}$; see Definition \ref{def: loose}.
		However, $h_{\Lambda_1', \epsilon_2}(\Lambda_2')$ is not loose \textit{in the complement of} $\Lambda_1'$ since $\Lambda_1'$  intersects the loose chart of $h_{\Lambda_1', \epsilon_2}(\Lambda_2')$. This completes the case $j = 2$. Note that we can extend the Legendrian isotopy of $\Lambda_1'$ back to $\Lambda_1$ to an ambient contact isotopy and hence  assume that $\Lambda_1' = \Lambda_1$.
		
		Now suppose that the $j-1$ case holds for some $j \ge 3$. So we have Weinstein homotoped $W$ to 
		$W_{sub} \cup H^n_{\Lambda_1} \cup \cdots \cup H^n_{\Lambda_k}$ (relabeling the Legendrians) such that $\coprod_{i=2}^{j-1} \Lambda_i$ is a loose link (but not loose in the complement of $\Lambda_1$).
		Again we take a Darboux ball $U_j$ that is disjoint from all the Legendrians and unlinked Legendrian unknots $S_j, T_j \subset U_j$. Then we form $\Lambda_1' := \Lambda_1 \sharp S_j, \Lambda_j' := \Lambda_j \sharp T_j$ using arcs $\gamma_1, \gamma_j$ that are disjoint from the other Legendrians. 
		Then we take sufficiently small $\epsilon_j$ (smaller than the previous $\epsilon_{j-1}$) and use the chart in $U_j$ to handle-slide $\Lambda_j'$ over $\Lambda_1'$ and get a new Legendrian 
		$h_{\Lambda_1'}(\Lambda_j')$.  
		Then by Proposition \ref{prop: handleslide} (and Remark \ref{rem: link_handleslide}), 
		$W_{sub} \cup H^n_{\Lambda_1'} \cup
		H^n_{\Lambda_2} \cup \cdots \cup H^n_{\Lambda_{j-1}} \cup H^n_{h_{\Lambda_1'}(\Lambda_j')} \cup H^n_{\Lambda_{j+1}} \cup \cdots \cup H^n_{\Lambda_k}$ is Weinstein homotopic to 
		$W_{sub} \cup H^n_{\Lambda_1'} \cup
		H^n_{\Lambda_2} \cup \cdots \cup H^n_{\Lambda_{j-1}} \cup H^n_{\Lambda_j'} \cup H^n_{\Lambda_{j+1}} \cup \cdots \cup H^n_{\Lambda_k}$ and hence to $W$. 
		As before, we can see explicitly that $h_{\Lambda_1'}(\Lambda_{j}')$ is loose in $\partial W_{sub}$ (but not in the complement of $\Lambda_1'$ which intersects its loose chart).
		Most importantly the loose chart of  $h_{\Lambda_1'}(\Lambda_{j}')$ is contained in $U_{j}$, which is disjoint from $\Lambda_2, \cdots, \Lambda_{j-1}$. Therefore  $h_{\Lambda_1'}(\Lambda_{j}')$ is loose in the complement of these Legendrians, which form a loose link by the induction hypothesis. So 
		$\Lambda_2 \coprod \cdots \coprod \Lambda_{j-1} \coprod h_{\Lambda_1'}(\Lambda_{j}')$ is also a loose link, which proves the $j$th inductive case. Again by applying an ambient contact isotopy to all the Legendrians, we can assume that $\Lambda_1' = \Lambda_1$. 
		\begin{figure}
			\centering
			\includegraphics[scale=0.2]{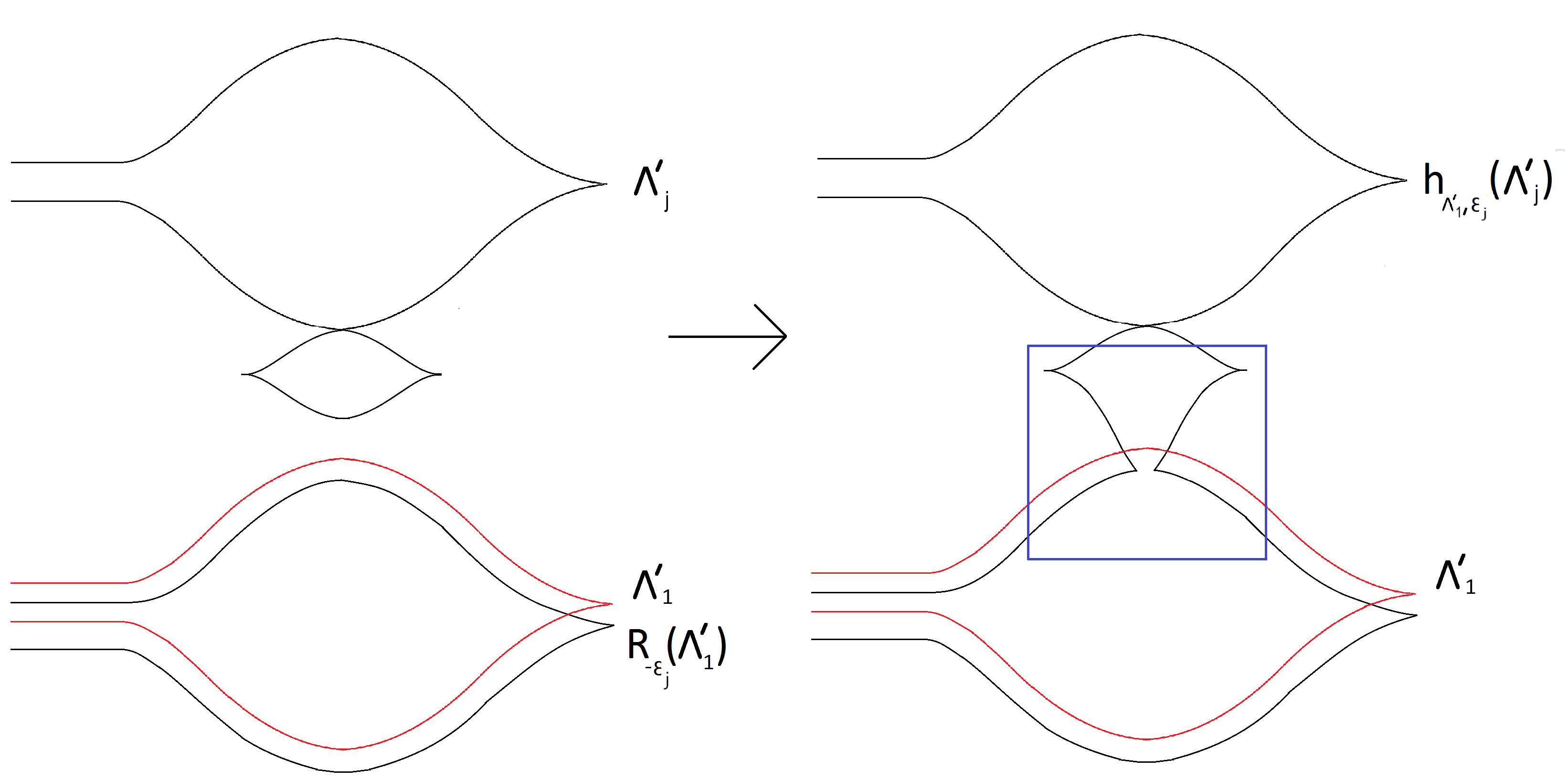}
			\caption{Left-hand figure: front projection of $\Lambda_j', \Lambda_{1}'$, and      $R_{-\epsilon_j}(\Lambda_1')$ in $U_j$; 
				right-hand figure:  front projection of $h_{\Lambda_1', \epsilon_j}(\Lambda_j')$ in $U_j$; the blue box is a loose chart of $h_{\Lambda_1', \epsilon_j}(\Lambda_j')$ in $U_j$.}
			\label{fig: handleslide}
		\end{figure}
	\end{proof}

	Now we give an example illustrating the entire procedure in Theorem \ref{thm: flexdomainhandle}.
	\begin{examples}
		The following example shows that $T^*S^n \natural T^*S^n \natural T^*S^n$, the boundary connected sum of three copies of $T^*S^n$, can be Weinstein homotoped to $W_{flex} \cup H^n$ for some flexible domain $W_{flex}$. We begin with the ``natural" presentation of $T^*S^n \natural T^*S^n \natural T^*S^n$ of the form
		$B^{2n} \cup H^n_{\Lambda_1} \cup H^n_{\Lambda_2} \cup H^n_{\Lambda_3}$, where $\Lambda_1, \Lambda_2, \Lambda_3$ are three unlinked Legendrian unknots in $(S^{2n-1}, \xi_{std})$. In Figure \ref{fig: example_three_connected_sum}, $\Lambda_1$ is in red, $\Lambda_2$ (and its image after handle-slides) is in black, and $\Lambda_3$ (and its image after handle-slides) is in blue. The top diagram in Figure \ref{fig: example_three_connected_sum} denotes the setup after one iteration of the construction; the Legendrians are now 
		$\Lambda_1, h_{\Lambda_1}(\Lambda_2), \Lambda_3$. The middle diagram in Figure \ref{fig: example_three_connected_sum} is the first part of the second iteration when we change $\Lambda_1$ to $\Lambda_1'$ and it bring it closer to $\Lambda_3$. The 
		bottom diagram in Figure \ref{fig: example_three_connected_sum} shows the three Legendrians $\Lambda_1, h_{\Lambda_1}(\Lambda_2), h_{\Lambda_1'}(\Lambda_3)$ after the second iteration of the construction, i.e. handle-sliding $\Lambda_3$ over $\Lambda_1'$.  Then $h_{\Lambda_1}(\Lambda_2), h_{\Lambda_1'}(\Lambda_3)$ form a loose link since $h_{\Lambda_1}(\Lambda_2)$ is a loose Legendrian and 
		$h_{\Lambda_1'}(\Lambda_3)$ is loose in the complement of $h_{\Lambda_1}(\Lambda_2)$. We take $W_{flex}$ to be $B^{2n} \cup H^n_{h_{\Lambda_1}(\Lambda_2)} \cup H^n_{h_{\Lambda_1'}(\Lambda_3)}$. So the original domain $T^*S^n \natural T^*S^n \natural T^*S^n$ is homotopic to $W_{flex} \cup H^n_{\Lambda_1'}$. Note that $h_{\Lambda_1}(\Lambda_2), h_{\Lambda_1'}(\Lambda_3)$ are not loose in the complement of $\Lambda_1'$, which intersects their loose charts. 
		For simplicity's sake,  $W_{flex}$ in this example is not actually $(T^*S^n\natural T^*S^n)_{flex}$; it will have the wrong intersection form (in some dimensions $n$) and so will not even be diffeomorphic to $T^*S^n\natural T^*S^n$. However it is possible to do the construction so that $W_{flex}$ is $(T^*S^n\natural T^*S^n)_{flex} \cup H^n$. 
		\begin{figure}
			\centering
			\includegraphics[scale=0.17]{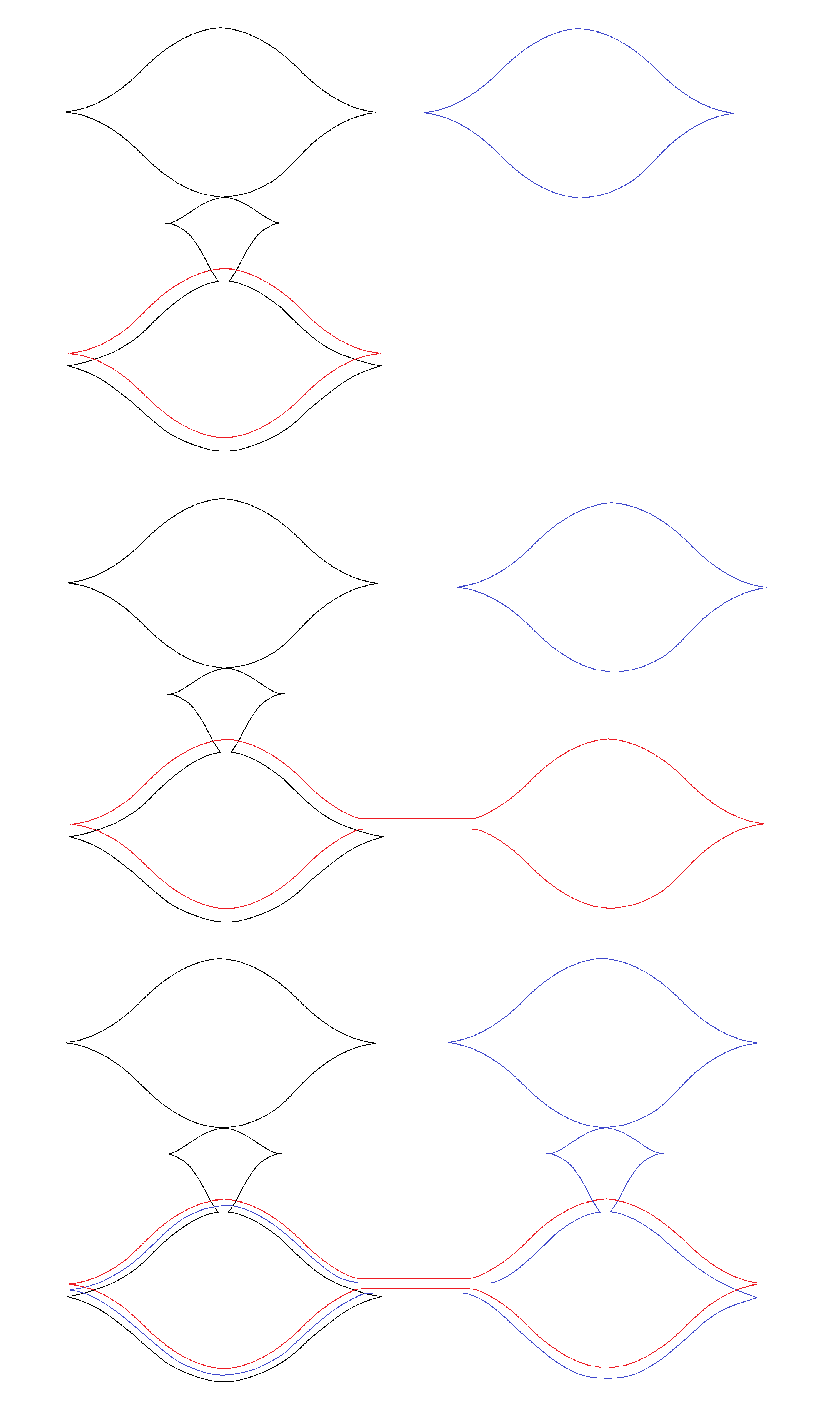}
			\caption{Theorem \ref{thm: flexdomainhandle} applied to  $T^*S^n_{std} \natural T^*S^n_{std} \natural T^*S^n_{std}$.}
			\label{fig: example_three_connected_sum}
		\end{figure}
		
		Although the order in which handles are attached does not affect
		the ambient domain (up to homotopy), it does affect which Weinstein subdomains are produced by a particular Weinstein presentation. To emphasize this, in Figure \ref{fig: cerf_diagram} we have depicted the Cerf diagram of the Weinstein homotopy for  $T^*S^n \natural T^*S^n \natural T^*S^n$ discussed above, i.e. the graph of critical values of the index n critical points of the Weinstein Morse functions $\phi_t$ over the parameter space $t \in [0,1]$. That is, if $p_i, i = 1,2,3,$ is the critical point with attaching sphere $\Lambda_i$  in the regular level set $(S^{2n-1}, \xi_{std})$, then the three line-graphs depict $\phi_t(p_i)$ for $t \in [0,1]$. In Figure \ref{fig: cerf_diagram}, we have labeled the graph of $\phi_t(p_i)$ by its attaching sphere. Handles are attached in order of the critical values of the corresponding critical points, from lowest to highest. At the beginning of the homotopy, $\phi_0(p_2), \phi_0(p_3)$ are greater than $\phi_0(p_1)$ since we need to handle-slide the $\Lambda_2, \Lambda_3$ handles over $\Lambda_1$. These handle-slide moments are depicted by the two vertical blue lines in Figure \ref{fig: cerf_diagram}. After the two handle-slides are performed, the attaching spheres of $p_2, p_3$ become $h_{\Lambda_1}(\Lambda_2), h_{\Lambda_1'}(\Lambda_3)$ respectively, as shown on the right-hand-side of Figure \ref{fig: cerf_diagram}. 
		Away from the handle-slide moments, the homotopy changes the Legendrian attaching spheres just by Legendrian isotopy. 
		Finally, the homotopy makes the critical value of $p_1$ greater than the critical values of $p_1, p_2$, which is possible by the second Weinstein homotopy move (see Section \ref{sssec: handleslides}). 
		As a result, the Weinstein domain $W_{flex}$ with attaching spheres $h_{\Lambda_1}(\Lambda_2), h_{\Lambda_1'}(\Lambda_3)$ is a sublevel set of $\phi_1$ and hence a Weinstein subdomain of  $T^*S^n \natural T^*S^n \natural T^*S^n$.

		\begin{figure}
			\centering
			\includegraphics[scale=0.27]{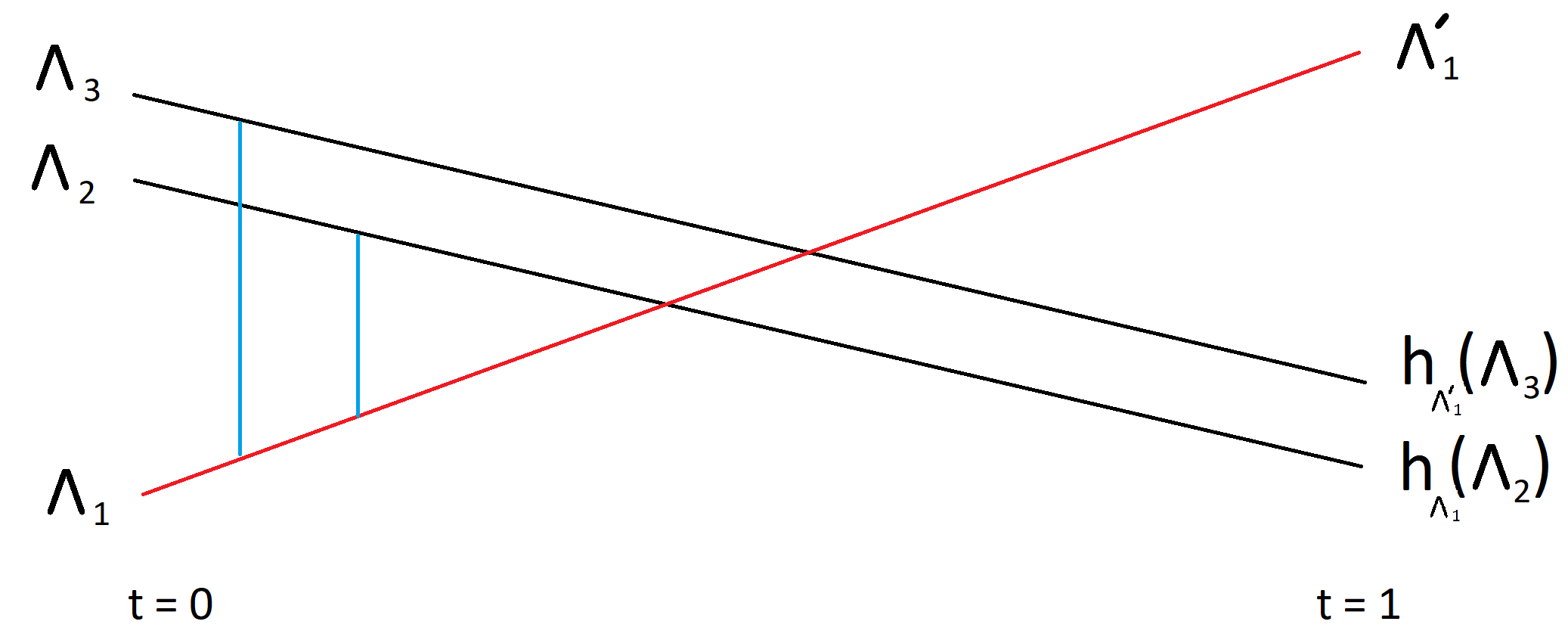}
			\caption{Cerf diagram of the Weinstein homotopy for $T^*S^n_{std} \natural T^*S^n_{std} \natural T^*S^n_{std}$.}
			\label{fig: cerf_diagram}
		\end{figure}
		
	\end{examples}

	Note that the Weinstein homotopy in Theorem \ref{thm: flexdomainhandle} involved just handle-slides. If we first create a pair of symplectically cancelling handles and then handle-slide, we can achieve better control over the topology of the flexible subdomain. This is the approach we will take in the following proof of Theorem \ref{thm: flexible_subdomain}, which shows that $W$ can be homotoped to $W_{flex} \cup C^{2n}$ for some smoothly trivial Weinstein cobordism $C^{2n}$ with two Weinstein handles. For example, this result shows that $T^*S^n \natural T^*S^n \natural T^*S^n$ can be Weinstein homotoped to $(T^*S^n\natural T^*S^n \natural T^*S^n)_{flex} \cup H^{n-1} \cup H^n_{\Lambda}$, where the last two handles are smoothly cancelling.

	\begin{proof}[Proof of Theorem \ref{thm: flexible_subdomain}]
		
		\begin{figure}
			\hspace*{-01.3cm}  
			\includegraphics[scale=0.2]{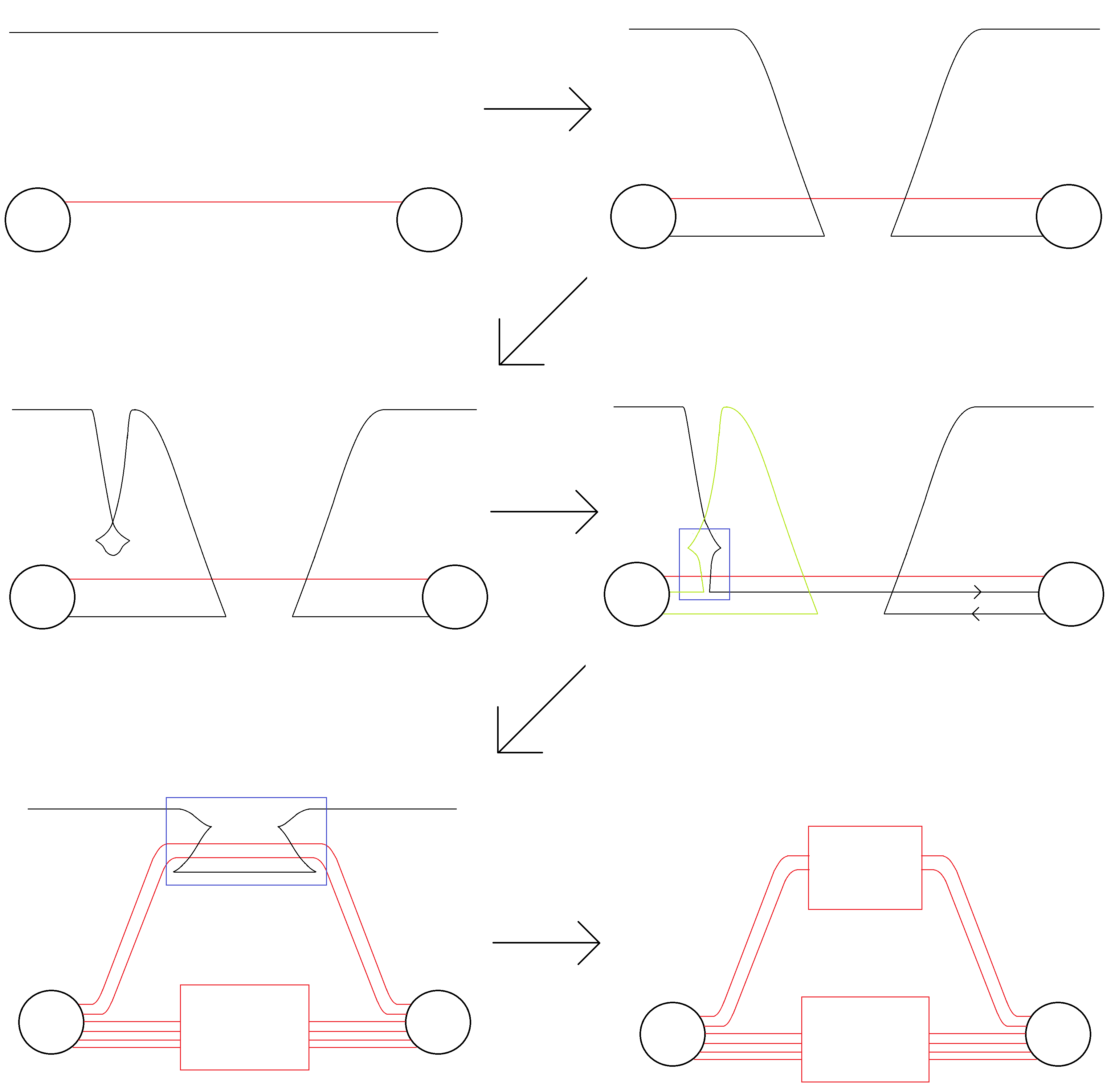}
			\caption{Front projections of $\Lambda_0$ (in red) and $\Lambda_i$ (in black) and their subsequent images under the moves in Theorems \ref{thm: flexible_subdomain} and \ref{thm: intersection points}; the blue box in the fourth, fifth diagrams is the loose chart of $h^2_{\Lambda_0}(\Lambda_i), \phi(h^2_{\Lambda_0}(\Lambda_i))$ respectively; the green portion of the Legendrian in the fourth diagram is the boundary of the Whitney disk between $h^2_{\Lambda_0}(\Lambda_i)$ and the belt sphere of $H^{n-1}$.}\label{fig: smoothisotopy}
		\end{figure}
		We will assume that $W = W_{sub} \cup 
		H^n_{\Lambda_1}\cup   \cdots \cup H^n_{\Lambda_k}$ for $k \ge 1$.  First, we attach a symplectically cancelling pair of index $n-1, n$ handles $H^{n-1}, H_{\Lambda_0}^n$ to $W$ in a small Darboux chart $B^{2n}$ so that $W = 
		W \natural (B^{2n} \cup H^{n-1} \cup H^n_{\Lambda_0}) = W_{sub}   \cup (H^{n-1} \cup H^n_{\Lambda_0}) \cup H^n_{\Lambda_1}\cup   \cdots \cup H^n_{\Lambda_k}$.
		Now we proceed as in the proof of Theorem \ref{thm: flexdomainhandle} with slight modifications. We first bring all the $\Lambda_i, i \ge 1,$ close to $\Lambda_0$ by taking $U_i$ in the proof of Theorem \ref{thm: flexdomainhandle} to be contained in $\partial B^{2n}$.
		The main difference from before is that now we do two handle-slides of $\Lambda_i, i \ge 1, $ over $\Lambda_0$, which produces the Legendrian $h_{\Lambda_0}^2(\Lambda_i)$. 
		Before doing the second handle-slide, we perform a Reidermeister move 
		which ensures that the second copy of $\Lambda_0$ that appears after this handle-slide crosses the belt sphere of $H^{n-1}$ in the opposite direction as the first copy. See the third and fourth diagrams in Figure \ref{fig: smoothisotopy} where $H^{n-1}$ is depicted as two circles as in the 1-dimensional case. Note that in the fourth diagram, the two branches of $h_{\Lambda_0}^2(\Lambda_i)$ enter the $n-1$ handle with opposite orientations, denoted by the arrows. In particular, $h_{\Lambda_0}^2(\Lambda_i)$ has algebraic intersection number zero with the belt sphere of $H^{n-1}$. 
		As in Theorem \ref{thm: flexdomainhandle},
		$h_{\Lambda_0}^2(\Lambda_1) \coprod \cdots \coprod h_{\Lambda_0}^2(\Lambda_k)$ form a loose link; 
		more precisely,  the $i$th Legendrian is loose in the complement of the previous $(i-1)$ Legendrians, which implies that the link is loose.   
		Hence $W' := W_{sub}\cup H^{n-1} \cup 
		H^n_{h^2_{\Lambda_0}(\Lambda_1)}  \cup \cdots \cup
		H^n_{h^2_{\Lambda_0}(\Lambda_k)}$  is flexible and 
		$W = W' \cup H^n_{\Lambda_0}$. 
		
		Since the algebraic intersection number of $h^2_{\Lambda_0}(\Lambda_i)$ with the belt sphere of $H^{n-1}$ is zero, $n \ge 3$, and $\pi_1(\partial(B^{2n} \cup H^{n-1})) = 0$, we can use the Whitney trick to smoothly isotope $h^2_{\Lambda_0}(\Lambda_i)$ away from this  belt sphere.
		In fact, we can assume that this smooth isotopy is supported in $\partial(B^{2n} \cup H^{n-1})$. To see this, note that we can take the boundary of the Whitney disk to lie in this region; see the green portion of Legendrian in the fourth diagram of Figure \ref{fig: smoothisotopy}. This region is simply-connected and hence the Whitney disk also lies in this 
		region; so the isotopy is also supported in this region. 
		Since $n \ge 3$,  the Whitney disks will be generically disjoint for different $i$ and so we can  smoothly isotope the whole link $h^2_{\Lambda_0}(\Lambda_1) \coprod \cdots \coprod h^2_{\Lambda_0}(\Lambda_k)$ off the belt sphere of $H^{n-1}$ (again via an isotopy supported in $\partial( B^{2n} \cup H^{n-1})$).
		
		The Legendrian link $h^2_{\Lambda_0}(\Lambda_1) \coprod \cdots \coprod h^2_{\Lambda_0}(\Lambda_k)$ is loose and so
		the smooth isotopy can be approximated by a Legendrian isotopy. Since the smooth isotopy is supported in $\partial( B^{2n} \cup H^{n-1})$ and the Legendrians are loose in this region, the Legendrian isotopy is also supported in this region. 
		Let $\phi_t$ be the ambient contact isotopy inducing this Legendrian isotopy and supported in a small neighborhood of the Legendrian isotopy; in particular $\phi_t$ is also supported in $\partial( B^{2n} \cup H^{n-1})$. 
		Since $h_{\Lambda_0}^2(\Lambda_1) \coprod \cdots \coprod h_{\Lambda_0}^2(\Lambda_k)$ is a loose link, so is
		$\phi(h_{\Lambda_0}^2(\Lambda_1)) \coprod \cdots \coprod \phi(h_{\Lambda_0}^2(\Lambda_k))$, where $\phi: = \phi_1$. 
		Furthermore, we can assume that this link is loose in the complement of $H^{n-1}$ and $\Lambda_0$ but not in the complement of $\phi(\Lambda_0)$. See the fifth diagram in Figure \ref{fig: smoothisotopy}. The upper Legendrian in black  is $\phi(h_{\Lambda_0}^2(\Lambda_i))$ and the blue box is its loose chart. The red Legendrian is $\phi(\Lambda_0)$. 
		This fifth diagram is purely schematic and is meant to demonstrate that $\phi(\Lambda_0)$ intersects the belt sphere of $H^{n-1}$ some number of times and is linked with $\phi(h^2_{\Lambda_0}(\Lambda_1))$ in some way such that 
		$\phi(\Lambda_0)$ intersects the loose chart of $\phi(h^2_{\Lambda_0}(\Lambda_i))$ (since $\Lambda_0$ intersected the loose chart of $h^2_{\Lambda_0}(\Lambda_i)$).

		Now we apply the contact isotopy $\phi$ to all attaching Legendrians; see the transition from the fourth to the fifth diagram in Figure \ref{fig: smoothisotopy}. 
		As a result, we get that 
		$W = W_{sub}\cup H^{n-1} \cup 
		H^n_{h^2_{\Lambda_0}(\Lambda_1)}  \cup \cdots \cup
		H^n_{h^2_{\Lambda_0}(\Lambda_k)} 
		\cup 
		H^n_{\Lambda_0}$ is Weinstein homotopic to 
		$W_{sub} \cup H^{n-1}  \cup 
		H^n_{\phi(h^2_{\Lambda_0}(\Lambda_1))}  \cup \cdots \cup
		H^n_{\phi(h^2_{\Lambda_0}(\Lambda_k))} 
		\cup H^n_{\phi(\Lambda_0)}$.
		The key point is that the latter presentation is Weinstein homotopic to 
		$W_{sub} \cup 
		H^n_{\phi(h^2_{\Lambda_0}(\Lambda_1))}  \cup \cdots \cup
		H^n_{\phi(h^2_{\Lambda_0}(\Lambda_k))} 
		\cup H^{n-1} \cup H^n_{\phi(\Lambda_0)}$ because we can attach the handles $H^n_{\phi(h^2_{\Lambda_0}(\Lambda_1))}  \cup \cdots \cup H^n_{\phi(h^2_{\Lambda_0}(\Lambda_k))}$ before $H^{n-1}$
		since $\phi(h^2_{\Lambda_0}(\Lambda_1))  \coprod \cdots \coprod \phi(h^2_{\Lambda_0}(\Lambda_k))$ is disjoint from the belt sphere of $H^{n-1}$.  
		Let $W''$ be the domain $W_{sub} \cup H^{n}_{\phi(h_{\Lambda_0}^2(\Lambda_1))}
		\cup \cdots \cup H^{n}_{\phi(h_{\Lambda_0}^2(\Lambda_k))}$ obtained by viewing $\phi(h^2_{\Lambda_0}(\Lambda_1))  \coprod \cdots \coprod \phi(h^2_{\Lambda_0}(\Lambda_k))$ as a Legendrian link in $\partial W_{sub}$. So $W$ is Weinstein homotopic to 
		$W'' \cup H^{n-1} \cup H^n_{\phi(\Lambda_0)}$. 
		We note that $W''$ is flexible since $\phi(h_{\Lambda_0}^2(\Lambda_1)) \coprod \cdots \coprod \phi(h_{\Lambda_0}^2(\Lambda_k))$ is loose in the complement of $H^{n-1}$. 
		
		Finally, we show that the Weinstein cobordism $W\backslash W'' =  H^{n-1}\cup H^n_{\phi(\Lambda_0)}$ is smoothly trivial. 
		Since $\phi$ is smoothly isotopic to the identity, $\phi(\Lambda_0)$ is smoothly isotopic to $\Lambda_0$ in $\partial (W_{sub} \cup H^{n-1})$.
		Since $\Lambda_0$ intersects the belt sphere of $H^{n-1}$ exactly once, this isotopy gives Whitney disks that cancel out all intersection points between $\phi(\Lambda_0)$ and the belt sphere of $H^{n-1}$ (except for one). Since $n \ge 3$, the Whitney disks will be generically disjoint from the link
		$\phi(h_{\Lambda_0}^2(\Lambda_1)) \coprod \cdots \coprod \phi(h_{\Lambda_0}^2(\Lambda_k))$. So  $\phi(\Lambda_0)$ can be smoothly isotoped in the complement of this link to a sphere that intersects the belt sphere of $H^{n-1}$  exactly once. This means that
		$\phi(\Lambda_0)$ can be smoothly isotoped in $\partial (W'' \cup H^{n-1})$ 
		to intersect this belt sphere  exactly once,  which proves that $W\backslash W'' = H^{n-1} \cup H^n_{\phi(\Lambda_0)}$ is smoothly trivial.

		Any almost symplectic structure on a smoothly trivial cobordism  can be deformed relative to the negative end to the product almost symplectic structure. In particular, $W, W''$ are almost symplectomorphic. Since $W''$ is flexible, by the uniqueness h-principle \cite{CE12} it is the flexibilization $W_{flex}$ of $W$. 
	\end{proof}
	Now we prove the 4-dimensional analog of Theorem \ref{thm: flexible_subdomain}. 
	\begin{proof}[Proof of Theorem \ref{thm: dim4_sh}]
		We take $V^4$ to be $W'$ from the proof of Theorem \ref{thm: flexible_subdomain} so that $W = V \cup H^2_{\Lambda_0}$. 
		Note that $V^4$ is obtained  by attaching a 1-handle and some 2-handles along $h^2_{\Lambda_0}(\Lambda_k)$  to $W_{sub}^4$. Each attaching knot for the 2-handles is stabilized in the complement of the previous ones; hence $V^4$ is a stabilized domain. 
		Finally, we note that $V^4$ is simply homotopy equivalent to $W^4 \cup H^1$. To see this, we consider the six-dimensional domain $V^4 \times B^2$; as can be seen explicitly, the attaching knots $h_{\Lambda_0}^2(\Lambda_k)$ are unknotted in the $B^6 \cup H^1$ region and hence can be smoothly isotoped to $\Lambda_k$. As a result,  this domain is diffeomorphic to $(W^4 \cup H^1) \times B^2$.  Here we do not use the Whitney trick directly since the region $B^6 \cup H^1$ is not simply-connected. 
	\end{proof}
	Using Theorem \ref{thm: flexible_subdomain}, we can prove Theorem \ref{thm: fewcrit}, our result relating $WCrit$ and $Crit$.  
	\begin{proof}[Proof of Theorem \ref{thm: fewcrit}]
		By Theorem \ref{thm: flexible_subdomain}, we can Weinstein homotope any Weinstein domain $W^{2n}, n \ge 3,$ to its flexiblization plus two smoothly cancelling handles of index $n-1, n$, i.e. to $W_{flex} \cup H^{n-1}\cup H^n_{\Lambda_1}$ where $\Lambda_1$ can be smoothly isotoped to intersect the belt sphere of $H^{n-1}$ exactly once. 
		For any smooth Morse function $f$ with critical points of index at most $n$ on $W$, there is a Weinstein homotopy of $W_{flex}$ to a Weinstein presentation with Weinstein Morse function $f$; see Theorem 14.1 of \cite{CE12}. Furthermore, if $f$ has $\partial W_{flex}$ as a regular level set, then this Weinstein homotopy is fixed on $\partial W_{flex}$ up to scaling. By Smale's handle-trading trick, 
		there exists such a smooth function on $W$ that minimizes the number of critical points, i.e. with $Crit(W)$ critical points, and so we can Weinstein homotope $W_{flex}$ to a Weinstein presentation with $Crit(W)$ critical points. Since this homotopy is fixed up to scaling on $\partial W_{flex}$, it extends to a Weinstein homotopy of $W_{flex} \cup H^{n-1} \cup H^n_{\Lambda_1}$, which is fixed up to scaling in $W\backslash W_{flex}$. In particular, this homotopy on $W_{flex} \cup H^{n-1} \cup H^n_{\Lambda_1}$ does not alter the number of critical points in $W\backslash W_{flex}$.   Combining the homotopy of $W$ to $W_{flex} \cup H^{n-1} \cup H^n_{\Lambda_1}$ and this second homotopy of $W_{flex} \cup H^{n-1} \cup H^n_{\Lambda_1}$ to a presentation with few critical points, we get a Weinstein homotopy of $W$ to a Weinstein presentation with $Crit(W) + 2$ critical points: $Crit(W)$ critical points in $W_{flex}$ and 2 critical points in $W\backslash W_{flex}$ due to the handles $H^{n-1}, H^{n}_{\Lambda_1}$. This proves the first claim in Theorem \ref{thm: fewcrit}.

		Now we prove the third claim in Theorem \ref{thm: fewcrit} about smoothly subcritical domains $W^{2n}$. If $W^{2n}$ is Weinstein subcritical, then $W^{2n}$ is flexible and so by the above discussion can be homotoped to a Weinstein presentation with $Crit(W)$ critical points, i.e. $WCrit(W) = Crit(W)$. 
		Conversely, suppose that $WCrit(W) = Crit(W)$ and $\pi_1(W) = 0$. 
		If $\pi_1(W) = 0$, the proof of Smale's h-cobordism theorem shows that $Crit(W)$ equals the number of generators and relations for integral homology; see Theorem 6.1 of \cite{smale_structure_manifolds}. Then any minimizing smooth Morse function on $W$ cannot have any critical points of index greater than $n-1$ since these critical points are algebraically unnecessarily; we can remove them and still have generators for integral homology since $H_n(W; \mathbb{Z}) = 0$ and $H_{n-1}(W; \mathbb{Z})$ is torsion-free for smoothly subcritical $W$. 
		Hence if $\pi_1(W) = 0$ and $WCrit(W) = Crit(W)$, then the minimal Weinstein presentation gives a minimal smooth presentation and so cannot have any critical points of index greater than $n-1$. Therefore $W$ is Weinstein subcritical. Finally, we note that if $WCrit(W^{2n}) \ne Crit(W^{2n})$, then $WCrit(W^{2n}) = Crit(W^{2n})+2$ since $WCrit(W^{2n}) \le Crit(W^{2n})+2$ by the first claim and  $WCrit(W^{2n}) \equiv Crit(W^{2n})+2 \mod 2$ by the Euler characteristic.
		
		Now we prove the smoothly critical case. Suppose that $\psi$ is a minimal smooth Morse function on $W$ with $k = Crit(W)$ critical points. By assumption, one of these critical points has index $n$ (and the rest of the critical points have index at most $n$). 
		By the previous discussion, we can assume that $\psi$ is a Weinstein Morse function on $W_{flex}$ and two other smoothly cancelling handles $H^{n-1}, H^n_{\Lambda_1}$ are attached to $W_{flex}$  to form $W$. The smooth isotopy from $\Lambda_1$ to cancelling position gives some number of Whitney disks in $\partial(W_{flex} \cup H^{n-1})$ pairing off all intersection points of $\Lambda_1$ and the belt sphere of $H^{n-1}$ (except for one intersection point). 
		
		We can suppose that the index $n$ critical point of $\psi$ on $W_{flex}$ is attached along a loose Legendrian $\Lambda_0$; so $W_{flex} = W_{flex}' \cup H^{n}_{\Lambda_0}$ and $W = W_{flex}' \cup H^{n-1}\cup H^n_{\Lambda_0} \cup H^n_{\Lambda_1}$. Note that $\Lambda_0$ is disjoint from the belt sphere of $H^{n-1}$ (since $H^{n-1}$ is attached after $H^n_{\Lambda_0}$). We view $\Lambda_1 \subset \partial (W_{flex}' \cup H^{n-1})$ by taking any Legendrian in $\partial (W_{flex}' \cup H^{n-1})$ that is isotopic to $\Lambda_1$ in $\partial (W_{flex}' \cup H^{n-1} \cup H^n_{\Lambda_0})$; in general, there will be many such Legendrians, which are non-isotopic in $\partial (W_{flex}' \cup H^{n-1})$. Since $n \ge 3$, we can assume that the Whitney disks of $\Lambda_1$ in 
		$\partial(W_{flex} \cup H^{n-1})$ are disjoint from the belt sphere of $H^n_{\Lambda_0}$ and hence lie in $\partial(W_{flex}' \cup H^{n-1})$. In particular, $\Lambda_1$ can be smoothly isotoped in $\partial(W_{flex}' \cup H^{n-1})$ to intersect the belt sphere of $H^{n-1}$ in a single point. Furthermore, since the Whitney disks are disjoint from $\Lambda_0$ (since they are disjoint from its belt sphere), we can assume that this isotopy is supported away from $\Lambda_0$.  We  can also assume that this smooth isotopy of $\Lambda_1$ is the identity in a neighborhood of some point $x$ in $\Lambda_1$. We take an isotropic path $\gamma$ from $x$ to $\Lambda_0$ and also assume that the isotopy is the identity in a neighborhood of this path.

		Now we handle-slide $\Lambda_1$ over $\Lambda_0$ using the path $\gamma$. More precisely, we take the Legendrian connected sum of $\Lambda_1$ with  a Legendrian unknot near $\Lambda_0$ via the isotropic arc $\gamma$ and then handle-slide using a chart near this Legendrian unknot as in Theorem \ref{thm: flexdomainhandle}. We also do the handleslide so that the resulting Legendrian $h_{\Lambda_0}(\Lambda_1)$ is loose in $\partial (W_{flex}' \cup H^{n-1})$ (but not in the complement of $\Lambda_0$). Now we note that $h_{\Lambda_0}(\Lambda_1)$ can also be smoothly isotoped in 
		$\partial (W_{flex}' \cup H^{n-1})$ to a cancelling sphere that intersects the belt sphere of $H^{n-1}$ once. Namely, we can use exactly the same smooth isotopy that takes $\Lambda_1$ to a cancelling sphere. This is because $h_{\Lambda_0}(\Lambda_1)$ is topologically the connected sum of $\Lambda_0$ and $\Lambda_1$. Since the previous isotopy is supported away from $\Lambda_0$ and the path $\gamma$ used for the connected sum, we can extend it to the connected sum. Furthermore, $\Lambda_0$ is disjoint from the belt sphere of $H^{n-1}$ and so after the smooth isotopy,  $h_{\Lambda_0}(\Lambda_1)$ intersects this belt sphere once. 
		
		Since $h_{\Lambda_0}(\Lambda_1)$ is loose in $\partial (W_{flex}' \cup H^{n-1})$ and smoothly cancels $H^{n-1}$, we can symplectically cancel $H^{n-1}$ and $H^{n}_{h_{\Lambda_0}(\Lambda_1)}$. 
		Therefore $W_{flex}' \cup H^{n-1} \cup H^n_{\Lambda_0} \cup H^n_{h_{\Lambda_0}(\Lambda_1)}$ is Weinstein homotopic to $W_{flex}' \cup H^n_{\Lambda_0'}$. Here $\Lambda_0'$ is the Legendrian obtained by handle-sliding $\Lambda_0$ off the cancelling pair $H^{n-1} \cup H^n_{h_{\Lambda_0}(\Lambda_1)}$, i.e. $\Lambda_0'$ is the image of $\Lambda_0$ in $W_{flex}' = W_{flex}' \cup H^{n-1} \cup H^n_{h_{\Lambda_0}(\Lambda_1)}$. Since $W_{flex}'$ has a Weinstein presentation with $k-1$ critical points, $W_{flex}' \cup H^{n}_{\Lambda_0'}$ has a presentation with $k = Crit(W)$ critical points. This completes the proof since 
		$W = W_{flex}' \cup H^{n-1} \cup H^n_{\Lambda_0} \cup H^n_{\Lambda_1}$ is Weinstein homotopic to $W_{flex}' \cup H^{n-1} \cup H^n_{\Lambda_0} \cup H^n_{h_{\Lambda_0}(\Lambda_1)}$, which is homotopic to  $W_{flex}' \cup H^n_{\Lambda_0'}$. 
	\end{proof}

	The proof of Theorem \ref{thm: fewcrit} can be used to prove Corollary \ref{cor: non_loose_link}: all Legendrians in our Legendrian link can be made individually loose. 
	\begin{proof}[Proof of Corollary \ref{cor: non_loose_link}]
		The proof of Theorem \ref{thm: fewcrit} in the smoothly critical case shows that
		$W = W_{flex}' \cup H^{n-1} \cup H^n_{\Lambda_0} 
		\cup H^n_{ h_{\Lambda_0}(\Lambda_1)}$ where $\Lambda_0, h_{\Lambda_0}(\Lambda_1)$ are both loose; 
		$\Lambda_0$ is loose by assumption and $h_{\Lambda_0}(\Lambda_1)$ is loose because of the handle-slide. Combining $\Lambda_0$ with the attaching spheres of the n-handles of $W_{flex}' \cup H^{n-1}$
		(which form a loose link for some presentation), we get the desired result. For general $W$, we first add a pair of symplectically cancelling handles to $W_{flex}$ and then proceed as in the smoothly critical case. 
	\end{proof}
	
	Next we prove Theorem \ref{thm: intersection points} about the number of intersection points between the belt and attaching spheres of smoothly cancelling handles. 
	
	\begin{proof}[Proof of Theorem \ref{thm: intersection points}]
		By Theorem \ref{thm: flexdomainhandle}, we can assume that the  smoothly trivial Weinstein cobordism $W$ consists of two smoothly cancelling handles $H^{n-1}_1, H^n_{\Lambda_1}$, i.e. $\Lambda_1$ is smoothly isotopic to a Legendrian that intersects the belt sphere of $H^{n-1}_1$ in a single point.
		Now we follow the proof of Theorem \ref{thm: flexible_subdomain}. We first attach two cancelling handles $H^{n-1}_0, H^n_{\Lambda_0}$ in a small Darboux ball and do two handle-slides (of opposite orientations) of $\Lambda_1$ over $\Lambda_0$ so that the resulting Legendrian $h^2_{\Lambda_0}(\Lambda_1)$  is loose. 
		Then we use the contact isotopy $\phi$ to isotope  $h^2_{\Lambda_0}(\Lambda_1)$ away from the belt sphere of $H^{n-1}_{0}$.  
		The result is $W = H^{n-1}_0 \cup H^{n-1}_1 \cup H^n_{\phi(h^2_{\Lambda_0}(\Lambda_1))} \cup 
		H^n_{\phi(\Lambda_0)}$; see the fifth diagram in Figure \ref{fig: smoothisotopy}. The key observation is that this local diagram is independent of $\Lambda_1$ since all isotopies were done near $H^{n-1}_0\cup H^n_{\Lambda_0}$. In particular, let $C_n$ be the number of times that  $\phi(\Lambda_0)$ intersects the belt sphere of $H^{n-1}_0$; in Figure \ref{fig: smoothisotopy}, this number is $5$ but since we do not compute this isotopy $\phi$ explicitly we do not know the exact number. 
		
		Next we note that Legendrian $\phi(h_{\Lambda_0}^2(\Lambda_1))$ is still smoothly isotopic to a Legendrian that intersects the belt sphere of $H^{n-1}_1$ in a single point. This is because  $\phi(h_{\Lambda_0}^2(\Lambda_1))$ is exactly the same as 
		$\Lambda_1$ except for a loose chart; see the blue box in the fifth diagram of Figure \ref{fig: smoothisotopy}. 
		Furthermore, we can assume that this smooth isotopy is supported away from $H^{n-1} \cup H^n_{\Lambda_0}$. Since $\phi(h_{\Lambda_0}^2(\Lambda_1))$ is loose, there is a contact isotopy $\psi$ taking it to a Legendrian that intersects the belt sphere of $H^n_1$ in one point; since
		$\phi(h_{\Lambda_0}^2(\Lambda_1))$ is loose away from $H^{n-1} \cup H^n_{\Lambda_0}$ and the smooth isotopy is supported away from this region, we can assume that this contact isotopy is also supported away from $H^{n-1}_0 \cup H^n_{\Lambda_0}$. In particular, $\psi(\phi( \Lambda_0))$ still intersects the belt sphere of $H^{n-1}_0$ in $C_n$ points. Finally, we handle-slide $\psi(\phi( \Lambda_0))$ over $\psi(\phi(h_{\Lambda_0}^2(\Lambda_1))$ and off $H^{n-1}_1$. This also does not change its geometric intersection number with  the belt sphere of $H^{n-1}_0$ since $\psi(\phi(h_{\Lambda_0}^2(\Lambda_1))$ is disjoint from this belt sphere. 
		We call the resulting Legendrian $\Lambda_0'$. Then $W = H^{n-1}_0 \cup H^n_{\Lambda_0'}$ and $\Lambda_0'$ intersects the belt sphere of $H^{n-1}_0$ exactly $C_n$ times as desired. 
		The Legendrian $\Lambda_0'$ is depicted in the sixth diagram of Figure \ref{fig: smoothisotopy}.  This diagram is also schematic and is mean to signify that $\Lambda_0'$ has an upper and lower part; the lower part of $\Lambda_0'$ is close to $H^{n-1}_0$ and is independent of $\Lambda_1$ while the upper part of $\Lambda_0'$ depends on $\Lambda_1$ (and hence on $W$). 
	\end{proof}
	
	Next we prove Corollary \ref{cor: no_aug} about the Chekanov-Eliashberg DGA's of certain Legendrians. We show that the vanishing of the Grothendieck group of any Weinstein ball, Corollary \ref{cor: Grothendieck_group}, 	 implies that the Chekanov-Eliashberg algebra of a Legendrian sphere $\Lambda^{n-1} \subset (S^{n-1}\times S^n, \xi_{std})$ that is primitive in homology has no finite-dimensional representations. 		
	\begin{proof}[Proof of Corollary \ref{cor: no_aug}]
		Let $X^{2n}:= B^{2n}_{std}\cup H^{n-1} \cup H^n_{\Lambda}$. Since $[\Lambda] =1 \in H_{n-1}(S^{n-1}\times S^n; \mathbb{Z}) \cong \mathbb{Z}$, $H^n(X^{2n}; \mathbb{Z}) = 0$ and so
		$K_0(\mathcal{W}(X)) = 0$ by Corollary \ref{cor: Grothendieck_group}. Let $C^n \subset X^{2n}$ be the co-core of $H^n_{\Lambda}$. 
		Since $C^n$ is the only index $n$ co-core for $X^{2n}$, $C^{n}$ generates $\mathcal{W}(X)$ and so 
		$D^b\mathcal{W}(X):= H^0(Tw(Fuk(X)))$ is equivalent to 
		$H^0(Tw(CW(C,C)))$, where we treat $CW(C,C)$ is an $A_\infty$-category with one object.
		By \cite{BEE12}, $CW(C,C)$ is quasi-isomorphic to $CE(\Lambda)$ and hence $D^b\mathcal{W}(X)$ is exact equivalent to $H^0(Tw(CE(\Lambda) ))$. 
		
		Suppose that $CE(\Lambda)$ has a DGA map to $\mbox{Mat}(n, \mathbb{K})$.
		Then there is an $A_\infty$-functor 
		$Tw(CE(\Lambda)) \rightarrow Tw(\mbox{Mat}(n, \mathbb{K}))$ and an exact functor
		$H^0(Tw(CE(\Lambda)) \rightarrow H^0(Tw(\mbox{Mat}(n, \mathbb{K})))$ taking 
		$CE(\Lambda)$ to $\mbox{Mat}(n, \mathbb{K})$ (considered as twisted complexes). 
		Let $D(Mat(n,\mathbb{K}))$ denote the classical derived category of $Mat(n, \mathbb{K})$-modules and 
		$D_\infty(Mat(n,\mathbb{K}))$ its $A_\infty$ analog, i.e. the homotopy category of $A_\infty$-modules over $Mat(n,\mathbb{K})$. 			 
		There is an embedding $D(Mat(n,\mathbb{K})) \rightarrow D_\infty(Mat(n,\mathbb{K}))$; see \cite{keller_infinity}. Since $H^0(Tw(\mbox{Mat}(n, \mathbb{K})))$ is equivalent to  the subcategory of  $D_\infty(Mat(n,\mathbb{K}))$ generated by the free module $Mat(n,\mathbb{K})$ and since the exact subcategory
		$D Mat(n,\mathbb{K})$ contains this free module, 		 $H^0(Tw(\mbox{Mat}(n, \mathbb{K})))$ is also equivalent to the subcategory of 
		$D Mat(n,\mathbb{K})$ generated by the free module $Mat(n,\mathbb{K})$. This subcategory is an exact subcategory of
		$D^b Proj(Mat(n,\mathbb{K}))$, the bounded derived category  of 
		projective 
		$Mat(n,\mathbb{K})$-modules. In summary, there is an exact functor
		$D^b\mathcal{W}(X) \rightarrow D^b Proj(Mat(n,\mathbb{K}))$ taking 
		the co-core $C^n$  to the free module $Mat(n,\mathbb{K})$. 
		This functor induces a map of Grothendieck groups 
		$K_0(\mathcal{W}(X) )\rightarrow 
		K_0(D^b(Proj(Mat(n,\mathbb{K})))$, and the latter is just the usual Grothendieck group
		$K_0(Mat(n,\mathbb{K}))$ of projective $Mat(n,\mathbb{K})$-modules.
		It is well-known that $[\mbox{Mat}(n, \mathbb{K})] \in K_0(\mbox{Mat}(n, \mathbb{K})) \cong \mathbb{Z}$ is non-zero. Therefore $K_0(\mathcal{W}(X) )$ is also non-zero, which contradicts Corollary \ref{cor: Grothendieck_group}.
		Similarly, there are no DGA maps from $CE(\Lambda)$ to a commutative ring $R$ since $[R] \in K_0(R)$ is non-zero for commutative rings. 
	\end{proof}

	Now we prove Corollary \ref{cor: c0-close} concerning Legendrians that can be isotoped into neighborhoods of loose Legendrians. 
	
	\begin{proof}\ref{cor: c0-close}
		Consider a loose Legendrian sphere $A \subset (S^{n-1} \times S^n, \xi_{std})$ that is primitive in $H_n(S^{n-1} \times S^n; \mathbb{Z})$.  Let $B
		\subset (S^{n-1} \times S^n, \xi_{std})$ 
		be the stabilization of $A$, followed by a small Reeb push-off so that $A,B$ are disjoint and form a loose link. The stabilization is done so that $A,B$ are formally isotopic (and hence Legendrian isotopic). We can also assume that exist disjoint contact neighborhoods $U, V$ of $A,B$ respectively so that $A,B$ are loose in the complement of $V,U$ respectively. 	
		
		Since $A$ is loose, 
		$B^{2n}_{std} \cup H^{n-1} \cup H^n_{A}$ is Weinstein homotopic to $B^{2n}_{std}$. By attaching the handle $H^{n}_{A}$ using a neighborhood of $A$ contained in $U$,  we can assume that $B$ and its neighborhood $V$ are disjoint from the attaching neighborhood and hence extend to a Legendrian $B' \subset (S^{2n-1}, \xi_{std}) = \partial B^{2n}_{std}$ and a contact neighborhood $V'$ of $B'$. 
		Since $B$ is loose in the complement of $U$, its loose chart extends to $(S^{2n-1}, \xi_{std})$ and so $B'$ is loose. The belt sphere of $H^n_{A}$ is the standard Legendrian unknot and so  $B'$ is formally isotopic to the Legendrian unknot. Since $B'$ is loose,  it is the loose Legendrian unknot $\Lambda_{unknot, loose}$. 
		
		As in the statement of this result, consider a Legendrian sphere $\Lambda \subset (S^{2n-1}, \xi_{std})$ that can isotoped into a neighborhood of $\Lambda_{unknot, loose} = B'$
		and is primitive in $H_{n-1}(\Lambda_{unknot,loose}; \mathbb{Z})$; we can assume that this neighborhood is $V'$. Using the identification between $V' \subset (S^{2n-1}, \xi_{std})$ and 	
		$V \subset (S^{n-1} \times S^n, \xi_{std})$, $\Lambda \subset V'$ defines a Legendrian $\Lambda_0 \subset V \subset  (S^{n-1} \times S^n, \xi_{std})$. In particular, $\Lambda \subset (S^{2n-1}, \xi_{std})$ is obtained by trivially extending $\Lambda_0 \subset  (S^{n-1} \times S^n, \xi_{std})$ through the Weinstein cobordism from $B^{2n}_{std} \cup H^{n-1}$ to $B^{2n}_{std} = B^{2n}_{std} \cup H^{n-1} \cup H^n_{A}$ given by handle attachement along $A \subset (S^{n-1} \times S^n, \xi_{std})$. 
		Since $\Lambda_0 \subset V$, $A \subset  (S^{n-1} \times S^n, \xi_{std})$ is loose in the complement of $\Lambda_0$. Handle attachment along the loose Legendrian $A$ does not change the Chekanov-Eliashberg algebras of Legendrians, like $\Lambda_0$, that are disjoint from the loose chart of $A$; see \cite{BEE12, Lazarev_flexible_fillings}. Hence $CE(\Lambda_0), CE(\Lambda)$ are quasi-isomorphic; this is the key point where we use the fact that $\Lambda$ is in a neighborhood of $\Lambda_{unknot, loose}=B'$, which implies that $\Lambda_0$ is disjoint from the loose chart of $A$.  Without this condition, $CE(\Lambda_0), CE(\Lambda)$ could be completely different and in fact, $CE(\Lambda_0)$ could be zero while $CE(\Lambda)$ is arbitrary.

		The fact that $\Lambda$ is primitive in $H_{n-1}(\Lambda_{unknot, loose}; \mathbb{Z})$ implies that 
		$\Lambda_0 \subset  (S^{n-1} \times S^n, \xi_{std})$ is primitive in $H_{n-1}(B; \mathbb{Z})$ and hence primitive in $H_{n-1}(S^{n-1} \times S^n; \mathbb{Z})$. 	So
		$H^0(Tw (CE(\Lambda_0))$ is equivalent to $D^b \mathcal{W}(X)$, where $X^{2n}$ is the Weinstein ball  $B^{2n}_{std}\cup H^{n-1} \cup H^n_{\Lambda_0}$. Then as in Corollary \ref{cor: no_aug}, $CE(\Lambda_0)$ has no finite-dimensional representations or DGA maps to commutative rings. 		
		Since $CE(\Lambda)$ is quasi-isomorphic to $CE(\Lambda_0)$ by the previous paragraph, $CE(\Lambda)$ also has no finite-dimensional representations or DGA maps to a commutative ring. More precisely, this quasi-isomorphism implies that $H^0(Tw (CE(\Lambda)))$ and $H^0(Tw (CE(\Lambda_0))$ are equivalent and the rest of the proof is as in Corollary \ref{cor: c0-close}.		
	\end{proof}
	
	Combining Corollary \ref{cor: no_aug} with the existence of infinitely many exotic Weinstein balls, we conclude that there are infinitely many Legendrian spheres in $(S^{n-1} \times S^n, \xi_{std})$ or $(S^{2n-1}, \xi_{std})$ with no finite-dimensional representations; these Legendrians are also in a contact neighborhood of loose Legendrians and are primitive in their homology. 
	\begin{proof}[Proof of Corollary \ref{cor: exotic_leg_no_aug}]			
		
		McLean \cite{MM} showed that there are infinitely many exotic Weinstein balls $\Sigma^{2n}_{k}$ for each $n \ge 4$, distinguished by symplectic cohomology. 
		As explained in Example \ref{ex: ball}, 
		$WCrit(\Sigma_k^{2n}) = 3$ and so $\Sigma_k^{2n}$ can be presented as  $B^{2n}_{std}\cup H^{n-1} \cup H^n_{\Lambda_k}$ for some Legendrian $\Lambda_k \subset (S^{n-1}\times S^n, \xi_{std})$. Since $\Sigma^{2n}_k$ is a ball, $\Lambda_k$ is primitive in homology and so by Corollary \ref{cor: no_aug}, $CE(\Lambda_k)$ has 
		no finite-dimensional representations. 
		By \cite{BEE12}, the symplectic cohomology of $\Sigma_k^{2n}$ is isomorphic to the Hochschild homology of $CE(\Lambda_k)$ and hence $CE(\Lambda_k)$ are not acyclic and are different for different $k$, as desired. 
		
		Next we show that the Legendrians $\Lambda_k$ can be isotoped into a contact neighborhood of a loose Legendrian and are primitive in its homology class. To do so,  we observe that in fact any closed connected Legendrian $\Lambda^{n-1} \subset (S^{n-1} \times S^n, \xi_{std} )$ can be Legendrian isotoped into a neighborhood of the loose Legendrian $S^{n-1} \times \{p\}$ for any $p \in S^n$ (but is not necessarily primitive in its homology).
		 Since $\Lambda_k$ and $S^{n-1} \times \{p\}$ are both primitive in $H_{n-1}(S^{n-1} \times S^n; \mathbb{Z})$, $\Lambda_k$ is primitive in $H_{n-1}(S^{n-1}\times \{p\}; \mathbb{Z})$, as desired. To see that any Legendrian can be isotoped in a neighborhood of  $S^{n-1} \times \{p\}$, 		
		note that $(S^{n-1} \times S^n, \xi_{std}) = \partial(D^*S^{n-1} \times D^2)$, where $D^*S^{n-1}$ is the unit disk cotangent bundle of $S^{n-1}$. Let $\pi: D^*S^{n-1} \times D^2 \rightarrow D^2$ be the projection map. 	 
		Since $\Lambda^{n-1}$ has dimension $n-1$ and $T^*S^{n-1}$ can be retracted to an $n-1$-dimensional space $S^{n-1}$ and the ambient space $S^{n-1} \times S^n$ has dimension $2n-1$, by Thom's transversality theorem, there is a Legendrian isotopy of $\Lambda$ making it disjoint from $T^*S^{n-1} \times (1,0) = \pi^{-1}(1,0) \subset S^{n-1} \times S^n$; see \cite{eliashberg_mishachev}. Since $\Lambda^{n-1}$ is closed and $\pi$ is a closed map,  $\Lambda$ is actually disjoint from $\pi^{-1}(Op(1,0)) \cap (S^{n-1} \times S^n)$ and hence contained in $(S^{n-1} \times S^n)\backslash 	\pi^{-1}(Op(1,0))$. 
	There is a contact isotopy taking $(S^{n-1} \times S^n) \backslash\pi^{-1}(Op(1,0))$ to $\pi^{-1}(Op(-1,0) ) \cap S^{n-1} \times S^n$, which is a contact neighborhood of the loose Legendrian $S^{n-1}\times (-1,0) \subset T^*S^{n-1} \times D^2$. So there is a Legendrian isotopy of $\Lambda$ into a neighborhood of this Legendrian. Note that $S^{n-1}\times (-1,0)$ is of the form $S^{n-1} \times \{p\}$, for the appropriate $p \in S^{n}$, which proves the claim.

		For the second part of this corollary about Legendrians in $(S^{2n-1}, \xi_{std})$, we essentially reverse the procedure in the proof of Corollary \ref{cor: c0-close}. Take a loose Legendrian $A \subset  (S^{n-1} \times S^n, \xi_{std})$ disjoint from $\Lambda_k$ and loose in the complement of $\Lambda_k$.  
		Then $B^{2n}_{std}\cup H^{n-1} \cup H^n_{A}$ is flexible and hence Weinstein homotopic to $B^{2n}_{std}$. Since $\Lambda_k$ is disjoint from $A$, $\Lambda_k$ defines a Legendrian sphere $\Lambda'_k$ in $(S^{2n-1}, \xi_{std}) = \partial B^{2n}_{std}$. Since $A$ is loose in the complement of $\Lambda_k$, $CE(\Lambda'_k)$ is quasi-isomorphic to 
		$CE(\Lambda_k)$ by \cite{BEE12, Lazarev_flexible_fillings}, as discussed in the proof of Corollary \ref{cor: c0-close}. Therefore, 	$H^0(Tw (CE(\Lambda_k)))$ is equivalent to 
		$H^0(Tw (CE(\Lambda_k')))$ and so $\Lambda_k' \subset (S^{2n-1}, \xi_{std})$ has the same properties as $\Lambda_k\subset (S^{n-1} \times S^n, \xi_{std})$, i.e. $CE(\Lambda_k')$ has no finite-dimensional representations or DGA maps to a commutative ring and their Hochschild homology are different for different $k$. 			
		Finally, we observe that $\Lambda_k'$ is in a contact neighborhood of a loose Legendrian in $(S^{2n-1}, \xi_{std})$ and is primitive in its homology. By the previous paragraph,  $\Lambda_k \subset (S^{n-1}\times S^n, \xi_{std})$ is in a contact neighborhood of the loose Legendrian $S^{n-1} \times \{p\} $ and is primitive in its homology. The Legendrian $S^{n-1} \times \{p\}$ is isotopic to the Legendrian $B$  obtained by stabilizing $A$ and taking a small Reeb push-off; so we assume from the start that $\Lambda_k$ is in a neighborhood of $B$, is primitive in $H_{n-1}(B; \mathbb{Z})$, and is disjoint from $A$. So the extension $\Lambda_k'$ of $\Lambda_k$ is in a neighborhood of the extension $B'$ of $B$ to $(S^{2n-1}, \xi_{std})$ and is primitive in $H_{n-1}(B'; \mathbb{Z})$. 		
		Since $B$ is loose in the complement of $A$, $B' \subset (S^{2n-1}, \xi_{std})$ is a loose Legendrian, in fact the loose Legendrian unknot, which proves the claim.  
	\end{proof}

	\bibliographystyle{abbrv}
	\bibliography{sources}

\end{document}